\documentclass[11pt,reqno]{amsart}

\usepackage{amssymb}
\usepackage{amsmath,amsthm,amsfonts,amssymb,latexsym,mathrsfs,color,hyperref}
\usepackage{graphicx}
\usepackage{xspace}
\usepackage{multirow}
\usepackage{diagbox}
\usepackage{capt-of}
\usepackage{tikz}
\tikzset{
	level 1/.style = {sibling distance = 1.5cm},
	level 2/.style = {sibling distance = 0.8cm},
    level distance = 1.0 cm
}
\usetikzlibrary {decorations.pathmorphing}
\tikzstyle{snakeline} = [decorate, decoration={snake, amplitude=.4mm, segment length=2mm}]

\newtheorem{theorem}{Theorem}
\newtheorem{corollary}[theorem]{Corollary}
\newtheorem{proposition}[theorem]{Proposition}

\newtheorem{lemma}[theorem]{Lemma}
\newtheorem{definition}[theorem]{Definition}
\newtheorem{example}[theorem]{Example}
\newtheorem{problem}[theorem]{Problem}

\newcommand{\Plat}{{\rm Plat}}
\newcommand{\Dplat}{{\rm Dplat}}
\newcommand{\Pasc}{{\rm Pasc}}
\newcommand{\Rpd}{{\rm Rpd}}
\newcommand{\Eudd}{{\rm Eud}}
\newcommand{\Dasc}{{\rm Dasc}}
\newcommand{\Lap}{{\rm Lap}}
\newcommand{\dplat}{{\rm dplat}}
\newcommand{\pasc}{{\rm pasc}}

\newcommand{\lends}{{\rm exl}}
\newcommand{\mends}{{\rm exm}}
\newcommand{\rends}{{\rm exr}}
\newcommand{\eudd}{{\rm eud}}
\newcommand{\udd}{{\rm ud}}
\newcommand{\rpd}{{\rm rpd}}
\newcommand{\pd}{{\rm pd}}

\newcommand{\mm}{\mathcal{M}}

\newcommand{\Ddes}{{\rm Ddes\,}}

\newcommand{\Des}{{\rm Des\,}}

\newcommand{\Asc}{{\rm Asc\,}}

\newcommand{\m}{{\rm M}}

\newcommand{\dasc}{{\rm dasc\,}}

\newcommand{\plat}{{\rm plat\,}}
\newcommand{\ap}{{\rm ap\,}}
\newcommand{\lap}{{\rm lap\,}}

\newcommand{\ddes}{{\rm ddes\,}}
\newcommand{\des}{{\rm des\,}}

\newcommand{\mtn}{\mathcal{T}}
\newcommand{\mmn}{\mathcal{M}_{2n}}

\newcommand{\msn}{\mathfrak{S}_n}

\newcommand{\ms}{\mathfrak{S}}

\newcommand{\lrf}[1]{\lfloor #1\rfloor}

\newcommand{\mbn}{{\mathcal B}_n}

\newcommand{\mq}{\mathcal{Q}}
\newcommand{\mqn}{\mathcal{Q}_n}

\newcommand{\asc}{{\rm asc\,}}

\newcommand{\Stirling}[2]{\genfrac{\{}{\}}{0pt}{}{#1}{#2}}

\title{Stirling permutation codes}
\author[S.-M.~Ma]{Shi-Mei Ma}
\address{School of Mathematics and Statistics,
        Northeastern University at Qinhuangdao,
         Hebei 066000, P.R. China}
\email{shimeimapapers@163.com (S.-M. Ma)}
\author[H.~Qi]{Hao Qi}
\address{College of Mathematics and Physics, Wenzhou University, Wenzhou 325035, P.R. China}
\email{qihao@wzu.edu.cn(H.~Qi)}
\author{Jean Yeh}
\address{Department of Mathematics, National Kaohsiung Normal University, Kaohsiung 82444, Taiwan}
\email{chunchenyeh@nknu.edu.tw}
\author[Y.-N. Yeh]{Yeong-Nan Yeh}
\address{College of Mathematics and Physics, Wenzhou University, Wenzhou 325035, P.R. China}
\email{mayeh@math.sinica.edu.tw (Y.-N. Yeh)}
\subjclass[2010]{Primary 05A19; Secondary 05E05}
\begin{document}

\maketitle
\begin{abstract}
The development of the theories of the second-order Eulerian polynomials began with the works of
Buckholtz and Carlitz in their studies of an asymptotic expansion.
Gessel-Stanley introduced Stirling permutations and presented
combinatorial interpretations of the second-order Eulerian polynomials.
The Stirling permutations have been extensively studied by many researchers.
The aim of this paper is to give substantial generalizations of the second-order Eulerian polynomials.
We develop a general method for finding equidistributed statistics on Stirling permutations.
Firstly, we show that
the up-down-pair statistic is equidistributed with ascent-plateau statistic, and that the exterior up-down-pair statistic is equidistributed with
left ascent-plateau statistic.
Secondly, we introduce the Stirling permutation codes. Several equidistribution results follow from simple applications.
In particular, we find that six bivariable set-valued statistics are equidistributed on the set of Stirling permutations. As applications, we extend a classical result
independently established by Dumont and B\'ona.
Thirdly, we explore bijections among Stirling permutation codes, perfect matchings and trapezoidal words.
We then show the $e$-positivity of the enumerators of Stirling permutations
by left ascent-plateaux, exterior up-down-pairs and right plateau-descents.
In the final part, the $e$-positivity of the multivariate $k$-th order Eulerian polynomials is established, which
improves a result of Janson-Kuba-Panholzer and generalizes a recent result of Chen-Fu.
\bigskip

\noindent{\sl Keywords}: Stirling permutations; Set-valued statistics; $e$-Positivity; Symmetric functions
\end{abstract}
\date{\today}
\newpage
\tableofcontents
\section{Introduction}
\subsection{Notation and preliminaries}
\hspace*{\parindent}

The development of the theories of the second-order Eulerian polynomials began with the works of
Buckholtz~\cite{Buckholtz} and Carlitz~\cite{Carlitz65} in their studies of an asymptotic expansion.
Further developments continued with the contributions of Riordan~\cite{Riordan76}, Gessel-Stanley~\cite{Gessel78}, Dumont~\cite{Dumont80}, Park~\cite{Park1994}, B\'ona~\cite{Bona08},
Janson-Kuba-Panholzer~\cite{Janson11}, Haglund-Visontai~\cite{Haglund12} and Chen-Fu~\cite{Chen17,Chen21}.
The aim of this paper is to give substantial generalizations of these polynomials.

Put
\begin{equation}\label{eqr}
\mathrm{e}^{nx}=\sum_{r=0}^n\frac{(nx)^r}{r!}+\frac{(nx)^n}{n!}S_n(x).
\end{equation}
where $n$ is a positive integer and $x$ an arbitrary complex number.
The study of~\eqref{eqr} was initiated by Ramanujan, see~\cite{Ramanujan27}.
Buckholtz~\cite{Buckholtz} found that
$$S_n(x)=\sum_{r=0}^{k-1}\frac{1}{n^r}U_r(x)+O(n^{-k}),$$
where $$U_r(x)=(-1)^r\left(\frac{x}{1-x}\frac{\mathrm{d}}{\mathrm{d}x}\right)^r\frac{x}{1-x}=(-1)^r\frac{C_r(x)}{(1-x)^{2r+1}}.$$
Subsequently, Carlitz~\cite{Carlitz65} discovered that
$$C_n(x)=(1-x)^{2n+1}\sum_{k=0}^\infty \Stirling{n+k}{k}x^k,$$
where $\Stirling{n}{k}$ are the {\it Stirling numbers of the second kind}.
The polynomials $C_n(x)$ are now known as the {\it second-order Eulerian polynomials} and they satisfy the recurrence relation
$$C_{n+1}(x)=(2n+1)xC_n(x)+x(1-x)\frac{\mathrm{d}}{\mathrm{d}x}C_n(x),~C_0(x)=1.$$
The first few $C_n(x)$ are $$C_1(x)=x,C_2(x)=x+2x^2,C_3(x)=x+8x^2+6x^3.$$

In~\cite{Riordan76}, Riordan found that $C_n(x)$ are the enumerators of Riordan trapezoidal words of length $n$ by number of distinct numbers.
Subsequently, Gessel-Stanley~\cite{Gessel78} discovered that $C_n(x)$ are the descent polynomials of Stirling permutations in $\mqn$.
The Stirling permutations have been extensively studied by many researchers, see~\cite{Chen21,Dzhumadil14,Lin21,Liu21,Ma1902,Park1994} and references therein.

For $\mathbf{m}=(m_1,m_2,\ldots,m_n)\in \mathbb{N}^n$, let $\mathbf{n}=\{1^{m_1},2^{m_2},\ldots,n^{m_n}\}$ be a multiset,
where $i$ appears $m_i$ times. A multipermutation of $\mathbf{n}$ is a sequence of its elements.
Denote by $\ms_{\mathbf{n}}$ the set of multipermutations of $\mathbf{n}$.
We say that the multipermutation $\sigma$ of $\mathbf{n}$ is a {\it Stirling permutation} if $\sigma_s\geqslant\sigma_i$ as soon as $\sigma_i=\sigma_j$ and $i<s<j$.
Denote by $\mq_\mathbf{n}$ the set of Stirling permutations of $\mathbf{n}$.
When $m_1=\cdots=m_n=1$, the set $\mq_\mathbf{n}$ reduces to the symmetric group $\msn$,
which is the set of permutations of $[n]=\{1,2,\ldots,n\}$. When $m_1=\cdots=m_n=2$,
the set $\mq_\mathbf{n}$ reduces to $\mq_n$, which is the set of ordinary Stirling permutations of $[n]_2=\{1^2,2^2,\ldots,n^2\}$.
Except where explicitly stated, we always assume that all Stirling permutations belong to $\mqn$, and for $\sigma\in\mqn$, we set $\sigma_0=\sigma_{n+1}=0$.
For example, $$\mq_1=\{11\},~\mq_2=\{1122,1221,2211\}.$$
\begin{definition}
For $\sigma\in\ms_{\mathbf{n}}$, any entry $\sigma_{i}$  is called
\begin{itemize}
  \item [$(i)$] an {\it ascent} (resp.~{\it descent},~{\it plateau}) if $\sigma_{i}<\sigma_{i+1}$ (resp.~$\sigma_{i}>\sigma_{i+1}$,~$\sigma_{i}=\sigma_{i+1}$), where $i\in \{0,1,2,\ldots,m_1+m_2+\cdots+m_n\}$ and we set $\sigma_0=\sigma_{m_1+m_2+\cdots+m_n+1}=0$, see~\cite{Bona08,Gessel78};
  \item [$(ii)$] an {\it ascent-plateau} (resp.~{\it plateau-descent}) if $\sigma_{i-1}<\sigma_{i}=\sigma_{i+1}$ (resp.~$\sigma_{i-1}=\sigma_{i}>\sigma_{i+1}$ ), where $i\in\{2,3,\ldots,m_1+m_2+\cdots+m_n-1\}$, see~\cite{Ma15,Ma19};
  \item [$(iii)$] a {\it left ascent-plateau} if $\sigma_{i-1}<\sigma_{i}=\sigma_{i+1}$, where $i\in[m_1+m_2+\cdots+m_n-1]$ and $\sigma_0=0$, see~\cite{Ma15,Ma19};
  \item [$(iv)$] a {\it right plateau-descent} if $\sigma_{i-1}=\sigma_{i}>\sigma_{i+1}$, where $i\in\{2,3,\ldots,m_1+m_2+\cdots+m_n\}$ and we set $\sigma_{m_1+m_2+\cdots+m_n+1}=0$, see~\cite{Liu21,Ma1902}.
\end{itemize}
\end{definition}
Let $\asc(\sigma)$ (resp.~$\des(\sigma),~\plat(\sigma),~\ap(\sigma),~\pd(\sigma),~\lap(\sigma),~\rpd(\sigma)$) denotes the number of ascents (resp.~descents,~plateaux,~ascent-plateaux,~plateau-descents,~left ascent-plateaux,~right plateau-descents) of $\sigma$.
The reverse bijection $\sigma\rightarrow \sigma^r$ on $\mqn$ defined by $\sigma^r_i=\sigma_{2n+1-i}$
shows that
\begin{equation*}
\begin{split}
\sum_{\sigma\in \mqn}x^{\asc(\sigma)}&=\sum_{\sigma\in \mqn}x^{\des(\sigma)},\\
\sum_{\sigma\in \mqn}x^{\ap(\sigma)}&=\sum_{\sigma\in \mqn}x^{\pd(\sigma)},\\
\sum_{\sigma\in \mqn}x^{\lap(\sigma)}&=\sum_{\sigma\in \mqn}x^{\rpd(\sigma)}.
\end{split}
\end{equation*}

In~\cite{Bona08}, B\'ona introduced the plateau statistic $\plat$ and discovered that
$$C_n(x)=\sum_{\sigma\in\mqn}x^{\plat{(\sigma)}},$$
which leads to a remarkable equidistributed result on $\mqn$:
\begin{equation}\label{Bona}
\sum_{\sigma\in\mqn}x^{\asc{(\sigma)}}=\sum_{\sigma\in\mqn}x^{\plat{(\sigma)}}=\sum_{\sigma\in\mqn}x^{\des{(\sigma)}}.
\end{equation}
It should be noted that the plateau statistic has been considered by Dumont~\cite{Dumont80} in the name of the repetition statistic, and
it went unnoticed until it was independently studied by B\'ona.
A trivariate version of the second-order Eulerian polynomial is defined by
\begin{equation}\label{Cnxyz01}
C_n(x,y,z)=\sum_{\sigma\in\mqn}x^{\asc{(\sigma)}}y^{\des(\sigma)}z^{\plat{(\sigma)}}.
\end{equation}
Dumont~\cite[p.~317]{Dumont80} found that
\begin{equation}\label{Dumont80}
C_{n+1}(x,y,z)=xyz\left(\frac{\partial}{\partial x}+\frac{\partial}{\partial y}+\frac{\partial}{\partial z}\right)C_n(x,y,z),~C_1(x,y,z)=xyz.
\end{equation}
which implies that $C_n(x,y,z)$ is symmetric in the variables $x,y$ and $z$, and so~\eqref{Bona} holds.
The symmetry of $C_n(x,y,z)$ was rediscovered by Janson~\cite[Theorem~2.1]{Janson08} by constructing an urn model.
In~\cite{Haglund12}, Haglund and Visontai introduced a refinement of the polynomial $C_n(x,y,z)$ by indexing each ascent,
descent and plateau by the values where they appear.
Recently, using the theory of context-free grammars, Chen and Fu~\cite{Chen21}
found that $C_n(x,y,z)$ is $e$-positive, i.e.,
\begin{equation}\label{Cnxyz}
C_n(x,y,z)=\sum_{i+2j+3k=2n+1}\gamma_{n,i,j,k}(x+y+z)^{i}(xy+yz+zx)^{j}(xyz)^k,
\end{equation}
where the coefficient $\gamma_{n,i,j,k}$ equals the number of 0-1-2-3 increasing plane trees
on $[n]$ with $k$ leaves, $j$ degree one vertices and $i$ degree two vertices.

A {\it rooted tree} of order $n$ with the vertices labelled $1,2,\ldots,n$, is an increasing tree if the
node labelled $1$ is distinguished as the root, and the labels along
any path from the root are increasing.
An {\it increasing plane tree}, usually called {\it plane recursive tree},
is an increasing tree with the children of each vertex are linearly ordered (from left to right, say).
A {\it 0-1-2-$\cdots$-k increasing plane tree} on $[n]$ is an increasing plane tree with each
vertex with at most $k$ children. The {\it degree} of a vertex in a rooted tree is meant to be
the number of its children (sometimes called outdegree).
The {\it depth-first walk} of a
rooted plane tree starts at the root, goes first to the leftmost child of the root, explores that
branch (recursively, using the same rules), returns to the root, and continues with the next
child of the root, until there are no more children left.

The following definition will be used repeatedly.
\begin{definition}[{\cite{Dumont80}}]
A {\it ternary increasing tree} of size $n$ is an increasing plane tree with $3n+1$ nodes in which each
interior node has label and three children (a left child,
a middle child and a right child), and exterior nodes have no children and no labels.
\end{definition}
Let $\mtn_n$ denote the set of ternary increasing trees of size $n$, see Figure~\ref{Fig01} for instance. For any $T\in\mtn_n$,
it is clear that $T$ has exactly $2n+1$ exterior nodes.
Let $\lends(T)$ (resp.~$\mends(T)$,~$\rends(T)$) denotes the number of exterior left nodes (resp.~exterior middle nodes, exterior right nodes) in $T$.
Using a recurrence relation that is equivalent to~\eqref{Dumont80}, Dumont~\cite[Proposition~1]{Dumont80} found that
\begin{equation}\label{QnTn}
C_n(x,y,z)=\sum_{\sigma\in\mqn}x^{\asc(\sigma)}y^{\plat(\sigma)}z^{\des(\sigma)}=\sum_{T\in\mtn_n}x^{\lends(T)}y^{\mends(T)}z^{\rends(T)}.
\end{equation}
\subsection{Motivation and the organization of the paper}
\hspace*{\parindent}

\begin{figure}[tbp]
\begin{center}
\begin{tikzpicture} 
\node [circle,draw] {1}
    child {node [circle,draw] {2}
    		child { edge from parent }
    		child { edge from parent node [left] {2} node [right] {2}}
    		child { edge from parent }
    }
    child { edge from parent node [left] {1} node [right] {1} }
    child { edge from parent };
\end{tikzpicture}\quad 
\begin{tikzpicture}
\node [circle,draw] {1}
    child { edge from parent }
    child { node [circle,draw] {2}     		
    		child { edge from parent }
    		child { edge from parent node [left] {2} node [right] {2}}
    		child { edge from parent }
    		edge from parent node[left] {1} node[right] {1}
    }
    child { edge from parent };
\end{tikzpicture}\quad %
\begin{tikzpicture}
\node [circle,draw] {1}
    child { edge from parent }
    child { edge from parent node[left] {1} node[right] {1}}
    child { node [circle,draw] {2}     		
    		child { edge from parent }
    		child { edge from parent node [left] {2} node [right] {2}}
    		child { edge from parent }
    		edge from parent
    };
\end{tikzpicture}
\end{center}
\caption{The ternary increasing trees of order 2 encoded by $2211,1221,1122$,
and their $\operatorname{SP}$-codes are given by $((0,0),(1,1)),((0,0)(1,2))$ and $((0,0)(1,3))$, respectively .}
\label{Fig01}
\end{figure}
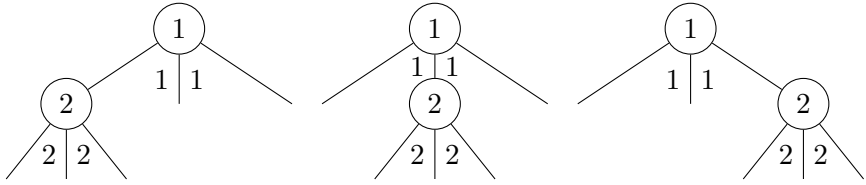
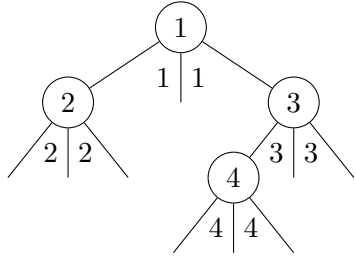
\begin{figure}[tbp]
\begin{center}
\begin{tikzpicture}
\node [circle,draw] {1}
    child { node [circle,draw] {2}     		
    		child { edge from parent}
    		child { edge from parent node [left] {2} node [right] {2}}
    		child { edge from parent}
    }
    child { edge from parent node[left] {1} node[right] {1}}
    child { node [circle,draw] {3}     		
 child { node [circle,draw] {4}     		
    		child { edge from parent}
    		child { edge from parent node [left] {4} node [right] {4}}
    		child { edge from parent }
    }
    		child { edge from parent node [left] {3} node [right] {3}}
child { edge from parent }
    		edge from parent
    };
\end{tikzpicture}
\end{center}
\caption{An order 3 ternary increasing tree encoded by $22114433$, and its $\operatorname{SP}$-code is $((0,0),(1,1),(1,3),(3,1))$.}
\label{Fig02}
\end{figure}

In~\cite{Janson08}, Janson gave a bijection between plane recursive trees and Stirling permutations,
which has previously been used by Koganov~\cite{Koganov96}. Subsequently, Janson-Kuba-Panholzer~\cite[Section~3]{Janson11} gave
a detailed proof of the bijection between $(k+1)$-ary increasing trees and $k$-Stirling permutations, which was independently introduced by Gessel~\cite[p46]{Park1994}.
Following the proof of~\cite[Theorem 1]{Janson11}, let $\phi$ be the bijection between ternary increasing trees and Stirling permutations that is defined as follows:
\begin{enumerate}
\item [$(i)$] Given $T\in\mtn_n$. Between the $3$ edges of $T$ going out from a node labelled $v$, we place $2$ integers $v$.
Now we perform the depth-first walk and code $T$ by the sequence of the labels visited as we go around $T$.
Let $\phi(T)$ be the code. In particular, the ternary increasing tree of order $1$ is encoded by $11$.
A ternary increasing tree of order $n$ is encoded by a string of $2n$ integers, where each of the labels $1,2,\ldots,n$ appears exactly $2$ times.
It is clear that for each $i\in[n]$, the elements occurring between the two
occurrences of $i$ are larger than $i$, since we can only visit nodes with higher labels. Hence the
code $\phi(T)$ is a Stirling permutation, see Figures~\ref{Fig01} and~\ref{Fig02} for illustrations;
\item [$(ii)$] The inverse of $\phi$ can be described as follows. Given $\sigma\in\mqn$. We proceed recursively
starting at step one by decomposing $\sigma$ as $u_11u_21u_3$, where the $u_i$'s are again Stirling permutations.
The smallest label in each $u_i$ is attached to the root node labelled $1$. One can recursively apply this procedure to each $u_i$ to obtain the tree representation, and
$\phi^{-1}(\sigma)$ is a ternary increasing tree.
 \end{enumerate}

Motivated by the work of B\'ona~\cite{Bona08}, Dumont~\cite{Chow08}, Haglund-Visontai~\cite{Haglund12} and Janson-Kuba-Panholzer~\cite[Section~3]{Janson11},
this paper is devoted to the following problem.
\begin{problem}\label{Problem01}
Give some applications of the bijection $\phi$ and develop a general method for finding equidistributed statistics on Stirling permutations.
\end{problem}

In Section~\ref{section2}, we introduce the up-down-pair statistic $\udd$ and the exterior up-down-pair statistic $\eudd$ on Stirling permutations, and we show that
$\udd$ is equidistributed with $\ap$ and $\eudd$ is equidistributed with $\lap$. Therefore, we get
$$\sum_{\sigma\in \mqn}x^{\ap(\sigma)}=\sum_{\sigma\in \mqn}x^{\udd(\sigma)}=\sum_{\sigma\in \mqn}x^{\pd(\sigma)},$$
\begin{equation}\label{laprpd}
\sum_{\sigma\in \mqn}x^{\lap(\sigma)}=\sum_{\sigma\in \mqn}x^{\eudd(\sigma)}=\sum_{\sigma\in \mqn}x^{\rpd(\sigma)}.
\end{equation}
In Section~\ref{section4}, by introducing the Stirling permutation code, we present various results concerning Problem~\ref{Problem01}.
In particular, in Theorems~\ref{thm13} and~\ref{thm14}, we present several bivariate generalizations of~\eqref{Bona}.
The last two identities in  Theorem~\ref{thm14} give the generalizations of~\eqref{Bona} and~\eqref{laprpd} simultaneously.
In Section~\ref{section402}, we establish bijections among $\operatorname{SP}$-codes, trapezoidal words and perfect matchings.
In Section~\ref{section5}, we show the $e$-positivity of the enumerators of Stirling permutations by $(\lap,\eudd,\rpd)$.
In Section~\ref{section6}, we show the $e$-positivity of the multivariate $k$-th order Eulerian polynomials,
which generalizes~\eqref{Cnxyz} and improves a classical result of Janson-Kuba-Panholzer~\cite{Janson11}.
\section{The ascent-plateau and up-down-pair statistics}\label{section2}
The number of elements in a set $C$ is called the cardinality of $C$, written $\#C$.
The {\it type $A$ Eulerian polynomials} $A_n(x)$~\cite{Hwang20}, the {\it type $B$ Eulerian polynomials} $B_n(x)$~\cite{Bre94}, the {\it ascent-plateau polynomials} (they aso also known as $1/2$-Eulerian polynomials) $M_n(x)$~\cite{Ma15,Savage1202}
and the {\it left ascent-plateau polynomials} $N_n(x)$~\cite{Ma15} can be respectively defined as follows:
\begin{equation*}
\begin{split}
A_n(x)&=\sum_{\pi\in\msn}x^{\des(\pi)},~B_n(x)=\sum_{\pi\in\mbn}x^{\operatorname{des}_B(\pi)},\\
M_n(x)&=\sum_{\sigma\in\mqn}x^{\ap(\sigma)},~
N_n(x)=\sum_{\sigma\in\mqn}x^{\lap(\sigma)},
\end{split}
\end{equation*}
where $\mbn$ denotes the the hyperoctahedral group of rank $n$,
$$\operatorname{des}_B(\pi)=\#\{i\in\{0,1,2,\ldots,n-1\}\mid\pi(i)>\pi({i+1})\},~{\text{and we set $\pi(0)=0$}}.$$
These polynomials share several similar properties,
including recursions~\cite{Gessel78,Ma15}, real-rootedness~\cite{Bona08,Haglund12}, combinatorial expansions~\cite{Chen21,Chow08,Lin21,Ma1902,Zhuang17} and asymptotic distributions~\cite{Hwang20}.
For convenience, we collect the recurrence relations of these polynomials:
\begin{equation*}
\begin{split}
A_{n+1}(x)&=(n+1)xA_{n}(x)+x(1-x)\frac{\mathrm{d}}{\mathrm{d}x}A_{n}(x),\\
B_{n+1}(x)&=(2nx+1+x)B_{n}(x)+2x(1-x)\frac{\mathrm{d}}{\mathrm{d}x}B_{n}(x),\\
M_{n+1}(x)&=(2nx+1)M_n(x)+2x(1-x)\frac{\mathrm{d}}{\mathrm{d}x}M_n(x),\\
N_{n+1}(x)&=(2n+1)xN_n(x)+2x(1-x)\frac{\mathrm{d}}{\mathrm{d}x}N_n(x),
\end{split}
\end{equation*}
with $A_0(x)=B_0(x)=C_0(x)=M_0(x)=N_0(x)=1$.
There are close connections among these polynomials (see~\cite{MaYeh17} for details):
\begin{equation}\label{Convo01}
\begin{split}
2^nA_n(x)&=\sum_{i=0}^n\binom{n}{i}N_i(x)N_{n-i}(x),\\
B_n(x)&=\sum_{i=0}^n\binom{n}{i}M_i(x)N_{n-i}(x).
\end{split}
\end{equation}

Let $\mqn^{(1)}$ be the set of Stirling permutations of the multiset $\{1,2^2,3^2,\ldots,n^2,(n+1)^2\}$, i.e.,
this multset has exactly one $1$ and two copies of $i$ for all $2\leqslant i\leqslant n+1$.
In particular, $$\mq_2^{(1)}=\{12233,12332,13322,33122,22133,22331,23321,33221\}.$$
By considering the position of the entry $1$ in a Stirling permutation $\sigma\in\mqn^{(1)}$, the following result immediately follows from~\eqref{Convo01}.
\begin{proposition}\label{thm01}
For $n\geqslant 1$, we have
$$2^nA_n(x)=\sum_{\sigma\in\mqn^{(1)}}x^{\lap(\sigma)},$$
$$B_n(x)=\sum_{\sigma\in\mqn^{(1)}}x^{\ap(\sigma)}.$$
\end{proposition}
\begin{definition}\label{def-uodown}
Let $\sigma\in\mqn$. An entry $\sigma_i$ is called an up-down-pair entry if $\sigma_{i-1}<\sigma_i=\sigma_j>\sigma_{j+1}$, where $i<j$.
The two equal entries $\sigma_i$ and $\sigma_j$ may appear arbitrarily far apart.
The up-down-pair statistic $\udd$ and the exterior up-down-pair statistic $\eudd$ are respectively defined as follows:
\begin{equation*}
\begin{split}
\udd(\sigma)&=\#\{i\in [2n-2]:~\text{$\sigma_i$ is an up-down-pair entry, and we set $\sigma_0=0$}\},\\
\eudd(\sigma)&=\#\{i\in [2n-1]:~\text{$\sigma_i$ is an up-down-pair entry, and we set $\sigma_0=\sigma_{2n+1}=0$}\}.
\end{split}
\end{equation*}
\end{definition}

\begin{example}
We have
\begin{equation*}
\begin{split}
\udd(123321)&=\udd(0123321)=2,\udd(331221)=\udd(0331221)=2,\\
\eudd(123321)&=\eudd(01233210)=3,\eudd(331221)=\eudd(03312210)=2.
\end{split}
\end{equation*}
\end{example}
The main result of this section is given as follows.
\begin{theorem}
We have
\begin{equation}\label{apud01}
\begin{split}
\sum_{\sigma\in\mqn}x^{\ap(\sigma)}&=\sum_{\sigma\in\mqn}x^{\udd(\sigma)},\\
\sum_{\sigma\in\mqn}x^{\lap(\sigma)}&=\sum_{\sigma\in\mqn}x^{\eudd(\sigma)}.
\end{split}
\end{equation}
\end{theorem}
\begin{proof}
Let $$M_n(x)=\sum_{\sigma\in\mqn}x^{\ap(\sigma)}=\sum_{i=0}^{n-1}M_{n,i}x^i,$$
$$N_n(x)=\sum_{\sigma\in\mqn}x^{\lap(\sigma)}=\sum_{i=1}^nN_{n,i}x^i.$$
Then the numbers $M_{n,i}$ and $N_{n,i}$ respectively satisfy the recurrence relations
\begin{equation}
\begin{split}
M_{n+1,i}&=(2i+1)M_{n,i}+(2n-2i+2)M_{n,i-1},\\
N_{n+1,i}&=2iN_{n,i}+(2n-2i+3)N_{n,i-1},
\end{split}
\end{equation}
with $M_{0,0}=N_{0,0}=1$ and $M_{0,i}=N_{0,i}=0$ if $i>0$, see~\cite{Ma15}.

Let $m_{n,i}=\#\{\sigma\in\mqn: \udd(\sigma)=i\}$.
It is clear that $m_{1,0}=M_{1,0}=1$, since $\udd(11)=\udd(011)=0$.
There are two ways to obtain an element $\sigma'\in\mq_{n+1}$ with $\udd(\sigma')=i$ from an element $\sigma\in\mqn$
by inserting two copies of $n$ into consecutive positions:
\begin{enumerate}
\item [($c_1$)] If $\udd(\sigma)=i$, then we can insert the two copies of $n$ before an up-down-pair entry or right after the second appearance of it.
 Moreover, we can insert the two copies of $n$ at the end of $\sigma$. This accounts for $(2i+1)m_{n,i}$ possibilities;
\item [($c_2$)] If $\udd(\sigma)=i-1$, then we insert the two consecutive copies of $n$ into one of the remaining $2n+1-(2(i-1)+1)=2n-2i+2$ positions.
This accounts for $(2n-2i+2)m_{n,i-1}$ possibilities.
 \end{enumerate}
Thus the numbers $m_{n,i}$ satisfy the same recursion and initial conditions as $M_{n,i}$, so they agree.

Define $n_{n,i}=\#\{\sigma\in\mqn: \eudd(\sigma)=i\}$.
Clearly, $n_{1,1}=N_{1,1}=1$, since $\eudd(0110)=1$.
Similarly,
there are two ways to obtain an element $\sigma'\in\mq_{n+1}$ with $\eudd(\sigma')=i$ from an element $\sigma\in\mqn$
by inserting two copies of $n$ into consecutive positions:
\begin{enumerate}
\item [($c_1$)] If $\eudd(\sigma)=i$, then we can insert the two copies of $n$ before an up-down-pair entry or right after the second appearance of it.
This accounts for $2in_{n,i}$ possibilities;
\item [($c_2$)] If $\eudd(\sigma)=i-1$, then we insert the two consecutive copies of $n$ into one of the remaining $2n+1-2(i-1)=2n-2i+3$ positions.
This accounts for $(2n-2i+3)n_{n,i-1}$ possibilities.
 \end{enumerate}
Thus the numbers $n_{n,i}$ satisfy the same recursion and initial conditions as $N_{n,i}$, so they agree.
\end{proof}
\section{Problem~\ref{Problem01} and the Stirling permutation code ($\operatorname{SP}$-code for short)}\label{section4}
Recall that a sequence $(e_1,e_2,\ldots,e_n)$ is an {\it inversion sequence} if $0\leqslant e_i<i$ for all $i\in [n]$.
It is well known that inversion sequences of length $n$ are in bijection with permutations in $\msn$.
As a dual of inversion sequence, by using the bijection $\phi$, we shall introduce a common code for
ternary increasing trees and Stirling permutations.

Recall that for any ternary increasing tree $T\in\mtn_n$, each interior node has label and three children (one at the left, a
middle child and a right child), and exterior nodes have no children and no labels.
For convenience, we introduce the following definition.
\begin{definition}
A simplified ternary increasing tree is a ternary increasing tree with no exterior nodes.
The {\it degree} of a vertex in a ternary increasing tree is meant to be the number of its children in the simplified ternary increasing tree.
\end{definition}
In fact, a simplified ternary increasing tree is the same as the ordinary ternary increasing tree, it is only a simplified version.
A node in a simplified ternary increasing tree with no children is called a leaf, and any interior node has at most three children (left child,
middle child or right child).
For example, Figure~\ref{Fig0003} gives the set of simplified ternary increasing trees of order $2$.
In the following discussion, a ternary increasing tree is always meant to be a simplified ternary increasing tree.
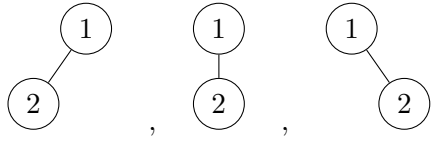
\begin{figure}
\begin{center}
\begin{tikzpicture}
[level 1/.style = {sibling distance = .7cm},
NONE/.style={edge from parent/.style={draw=none}}]
\node [circle,draw] {1}
    child {node [circle,draw] {2}}
    child [NONE] {}
    child [NONE] {};
\end{tikzpicture}
,
\begin{tikzpicture}
[level 1/.style = {sibling distance = .7cm},
NONE/.style={edge from parent/.style={draw=none}}]
\node [circle,draw] {1}
    child [NONE] {}
    child {node [circle,draw] {2}}
    child [NONE] {};
\end{tikzpicture}
,
\begin{tikzpicture}
[level 1/.style = {sibling distance = .7cm},
NONE/.style={edge from parent/.style={draw=none}}]
\node [circle,draw] {1}
    child [NONE] {}
    child [NONE] {}
    child {node [circle,draw] {2}};
\end{tikzpicture}
\caption{The simplified ternary increasing trees of order $2$.}
\label{Fig0003}
\end{center}
\end{figure}
A ternary increasing tree of size $n$ can be built up from the root $1$ by successively adding nodes $2,3,\ldots,n$.
Clearly, node $2$ is a child of the root $1$ and the root $1$ has at most three children (a left, a middle or a right child), see Figure~\ref{Fig0003} for illustartions. For $2\leqslant i\leqslant n$,
when node $i$ is inserted, we distinguish three cases:
\begin{enumerate}
\item [$(c_1)$] if it is the left child of a node $v\in [i-1]$, then the node $i$ is coded as $[v,1]$;
\item [$(c_2)$] if it is the middle child of a node $v\in [i-1]$, then the node $i$ is coded as $[v,2]$;
\item [$(c_3)$] if it is the right child of a node $v\in [i-1]$, then the node $i$ is coded as $[v,3]$.
 \end{enumerate}
Thus the node $i$ is coded as a $2$-tuple $(a_{i-1},b_{i-1})$, where $1\leqslant a_{i-1}\leqslant i-1$, $1\leqslant b_{i-1}\leqslant 3$ and $(a_i,b_i)\neq(a_j,b_j)$ for all $1\leqslant i<j\leqslant n-1$.
By convention, the root $1$ is coded as $(0,0)$.
Therefore, a ternary increasing tree of size $n$
corresponds naturally to a {\it build-tree code} $((0,0),(a_1,b_1),\ldots,(a_{n-1},b_{n-1}))$.
Using the bijection $\phi$ between
ternary increasing trees and Stirling permutations, one can see that
the build-tree code is the same as the {\it Stirling permutation code}, which is defined as follows.

\begin{definition}\label{def-BSP}
A $2$-tuples sequence $C_n=((0,0),(a_1,b_1),(a_2,b_2)\ldots,(a_{n-1},b_{n-1}))$ of length $n$ is a Stirling permutation code ($\operatorname{SP}$-code for short)
if $1\leqslant a_i\leqslant i$, $1\leqslant b_i\leqslant 3$ and $(a_i,b_i)\neq(a_j,b_j)$ for all $1\leqslant i<j\leqslant n-1$.
\end{definition}
Let $\operatorname{CQ}_n$ be the set of $\operatorname{SP}$-codes of length $n$.
In particular, $\operatorname{CQ}_1=\{(0,0)\}$ and $\operatorname{CQ}_2=\{(0,0)(1,1),~(0,0)(1,2),~(0,0)(1,3)\}$, see Figure~\ref{Fig01}.

\begin{theorem}\label{thmCQ}
The set $\operatorname{CQ}_n$ is in a natural bijection with the set $\mqn$, i.e., $\operatorname{CQ}_n\cong \mqn$.
\end{theorem}
\begin{proof}
For $n\geqslant 2$, there are three cases to obtain an element of $\mqn$ from an element $\sigma\in\mq_{n-1}$
by putting the two copies of $n$ between $\sigma_i$ and $\sigma_{i+1}$: $\sigma_i<\sigma_{i+1},~\sigma_i=\sigma_{i+1},~\sigma_i>\sigma_{i+1}$.
Set $\Gamma(11)=(0,0)$. When $n\geqslant 2$, the bijection $\Gamma: \mqn\rightarrow \operatorname{CQ}_n$ can be defined as follows:
\begin{enumerate}
\item [$(c_1)$] $\sigma_i<\sigma_{i+1}$ if and only if $(a_{n-1},b_{n-1})=(\sigma_{i+1},1)$;
\item [$(c_2)$] $\sigma_i=\sigma_{i+1}$ if and only if $(a_{n-1},b_{n-1})=(\sigma_{i+1},2)$;
\item [$(c_3)$] $\sigma_i>\sigma_{i+1}$ if and only if $(a_{n-1},b_{n-1})=(\sigma_{i},3)$.
 \end{enumerate}
See Figure~\ref{Fig01} and Example~\ref{exsig} for illustrations.
\end{proof}
\begin{example}\label{exsig}
Given $\sigma=551443312662\in\mq_5$.
We give the procedure of creating its $\operatorname{SP}$-code:
\begin{align*}
11&\Leftrightarrow (0,0),\\
1122&\Leftrightarrow (0,0)(1,3),\\
133122&\Leftrightarrow (0,0)(1,3)(1,2),\\
14433122&\Leftrightarrow (0,0)(1,3)(1,2)(3,1),\\
5514433122&\Leftrightarrow (0,0)(1,3)(1,2)(3,1)(1,1),\\
551443312662&\Leftrightarrow (0,0)(1,3)(1,2)(3,1)(1,1)(2,2).
\end{align*}
Thus $\Gamma(\sigma)=(0,0)(1,3)(1,2)(3,1)(1,1)(2,2)$. Conversely, we get $\Gamma^{-1}(\Gamma(\sigma))=\sigma$.
\end{example}

For $\sigma\in\mqn$, let
\begin{align*}
\operatorname{Asc}(\sigma)&=\{\sigma_i\mid \sigma_{i-1}<\sigma_{i}\},\\
\operatorname{Plat}(\sigma)&=\{\sigma_i\mid \sigma_i=\sigma_{i+1}\},\\
\operatorname{Des}(\sigma)&=\{\sigma_i\mid \sigma_i>\sigma_{i+1}\},\\
\operatorname{Lap}(\sigma)&=\{\sigma_i\mid \sigma_{i-1}<\sigma_{i}=\sigma_{i+1}\},\\
\operatorname{Rpd}(\sigma)&=\{\sigma_i\mid \sigma_{i-1}=\sigma_i>\sigma_{i+1}\},\\
\operatorname{Eud}(\sigma)&=\{\sigma_i\mid \sigma_{i-1}<\sigma_i=\sigma_j>\sigma_{j+1},~i<j\},\\
\operatorname{Dasc}(\sigma)&=\{\sigma_i\mid \sigma_{i-1}<\sigma_{i}<\sigma_{i+1}\},\\
\operatorname{Dplat}(\sigma)&=\{\sigma_i\mid \sigma_{i-1}>\sigma_i=\sigma_{i+1}\},\\
\operatorname{Ddes}(\sigma)&=\{\sigma_i\mid \sigma_{i-1}>\sigma_i>\sigma_{i+1}\},\\
\operatorname{Pasc}(\sigma)&=\{\sigma_i\mid \sigma_{i-1}=\sigma_i<\sigma_{i+1}\},\\
\operatorname{Apd}(\sigma)&=\{\sigma_i\mid \sigma_{i-1}<\sigma_i=\sigma_{i+1}>\sigma_{i+1}\},\\
\operatorname{Uu}(\sigma)&=\{\sigma_i\mid \sigma_{i-1}<\sigma_i=\sigma_{j}<\sigma_{j+1},~i<j\},\\
\operatorname{Dd}(\sigma)&=\{\sigma_i\mid \sigma_{i-1}>\sigma_i=\sigma_{j}>\sigma_{j+1},~i<j\}\\
\end{align*}
denote the sets of ascents, plateaux, descents, left ascent-plateaux, right plateau-descents, exterior up-down-pais, double ascents, descent-plateaux, double descents, plateau-ascents, ascent-plateau-descents, up-up-pairs and down-down-pairs of $\sigma$, respectively.
We use $\dasc(\sigma)$, $\dplat(\sigma)$, $\ddes(\sigma)$, $\pasc(\sigma)$, $\operatorname{apd}(\sigma)$, $\operatorname{uu}(\sigma)$ and $\operatorname{dd}(\sigma)$ to denote the
number of double ascents, descent-plateaux, double descents, plateau-ascents, ascent-plateau-descents, up-up-pairs and down-down-pairs of $\sigma$, respectively, i.e., $\dasc(\sigma)=\#\Dasc(\sigma),~\dplat(\sigma)=\#\Dplat(\sigma),~\ddes(\sigma)=\#\Ddes(\sigma)$, $\pasc(\sigma)=\#\Pasc(\sigma)$,
$\operatorname{apd}(\sigma)=\#\operatorname{Apd}(\sigma)$, $\operatorname{uu}(\sigma)=\#\operatorname{Uu}(\sigma)$ and $\operatorname{dd}(\sigma)=\#\operatorname{Dd}(\sigma)$.

Using the bijection $\phi$, it is easily seen that the set-valued statistics on Stirling permutations listed in Table~\ref{Table1}
correspond to the given set-valued statistics on $\operatorname{SP}$-codes. We illustrate
these correspondences in Example~\ref{example10}.
By Table~\ref{Table1}, a large number of properties of Stirling permutations can be easily deduced.
\begin{table}[h!]
  \begin{center}
    \begin{tabular}{l|c} 
      \textbf{Statistics on Stirling permutation} & \textbf{Statistics on $\operatorname{SP}$-code}\\
      \hline
      $\Asc$ (ascent) & $[n]-\{a_i\mid (a_i,1)\in C_n\}$ \\
      $\Plat$ (plateau) & $[n]-\{a_i\mid (a_i,2)\in C_n\}$ \\
      $\Des$ (descent)& $[n]-\{a_i\mid (a_i,3)\in C_n\}$ \\
    $\Lap$ (left ascent-plateau)& $[n]-\{a_i\mid (a_i,1)~\text{or}~ (a_i,2)\in C_n\}$ \\
    $\Rpd$ (right plateau-descent)& $[n]-\{a_i\mid (a_i,2)~\text{or}~ (a_i,3)\in C_n\}$ \\
   $\Eudd$ (exterior up-down-pair)& $[n]-\{a_i\mid (a_i,1)~\text{or}~ (a_i,3)\in C_n\}$ \\
    $\Dasc$ (double ascent)& $\{a_i\mid (a_i,1)\notin C_n~\&~ (a_i,2)\in C_n\}$ \\
    $\Dplat$ (descent-plateau)& $\{a_i\mid (a_i,1)\in C_n~\&~ (a_i,2)\notin C_n\}$ \\
   $\Ddes$ (double descent)& $\{a_i\mid (a_i,2)\in C_n~\&~ (a_i,3)\notin C_n\}$ \\
  $\Pasc$ (plateau-ascent)& $\{a_i\mid (a_i,2)\notin C_n~\&~ (a_i,3)\in C_n\}$ \\
  $\operatorname{Apd}$ (ascent-plateau-descent)& $\{a_i\mid (a_i,1)\notin C_n~\&~ (a_i,2)\notin C_n~\&~ (a_i,3)\notin C_n\}$\\
  $\operatorname{Uu}$ (up-up-pair)& $\{a_i\mid (a_i,1)\notin C_n~\&~ (a_i,3)\in C_n\}$\\
  $\operatorname{Dd}$ (down-down-pair)& $\{a_i\mid (a_i,1)\in C_n~\&~ (a_i,3)\notin C_n\}$
    \end{tabular}
  \end{center}
\caption{The correspondences of statistics on Stirling permutations and $\operatorname{SP}$-codes}
\label{Table1}
\end{table}

\begin{example}\label{example10}
Let $\sigma=77441223315665$. The corresponding $\operatorname{SP}$-code is given by $$C_7=(0,0)(1,2)(2,3)(1,1)(1,3)(5,2)(4,1).$$
Then we have
\begin{align*}
\operatorname{Asc}(\sigma)&=[7]-\{a_i\mid (a_i,1)\in C_7\}=\{2,3,5,6,7\},\\
\operatorname{Plat}(\sigma)&=[7]-\{a_i\mid (a_i,2)\in C_7\}=\{2,3,4,6,7\},\\
\operatorname{Des}(\sigma)&=[7]-\{a_i\mid (a_i,3)\in C_7\}=\{3,4,5,6,7\},\\
\operatorname{Lap}(\sigma)&=[7]-\{a_i\mid (a_i,1)~\text{or}~ (a_i,2)\in C_7\}=\{2,3,6,7\},\\
\operatorname{Rpd}(\sigma)&=[7]-\{a_i\mid (a_i,2)~\text{or}~ (a_i,3)\in C_7\}=\{3,4,6,7\},\\
\operatorname{Eud}(\sigma)&=[7]-\{a_i\mid (a_i,1)~\text{or}~ (a_i,3)\in C_7\}=\{3,5,6,7\},\\
\operatorname{Dasc}(\sigma)&=\{a_i\mid (a_i,1)\notin C_7~\&~ (a_i,2)\in C_7\}=\{5\},\\
\operatorname{Dplat}(\sigma)&=\{a_i\mid (a_i,1)\in C_7~\&~ (a_i,2)\notin C_7\}=\{4\},\\
\operatorname{Ddes}(\sigma)&=\{a_i\mid (a_i,2)\in C_7~\&~ (a_i,3)\notin C_7\}=\{5\},\\
\operatorname{Pasc}(\sigma)&=\{a_i\mid (a_i,2)\notin C_7~\&~ (a_i,3)\in C_7\}=\{2\},\\
\operatorname{Apd}(\sigma)&=\{a_i\mid (a_i,1)\notin C_7~\&~ (a_i,2)\notin C_7~\&~ (a_i,3)\notin C_7\}=\{3,6\},\\
\operatorname{Uu}(\sigma)&=\{a_i\mid (a_i,1)\notin C_7~\&~ (a_i,3)\in C_7\}=\{2\},\\
\operatorname{Dd}(\sigma)&=\{a_i\mid (a_i,1)\in C_7~\&~ (a_i,3)\notin C_7\}=\{4\}.
\end{align*}
\end{example}

The following two results give several generalizations of~\eqref{Bona}, which can be proved by switching some $2$-tuples in the corresponding $\operatorname{SP}$-codes.
\begin{theorem}\label{thm13}
The six bivariable set-valued statistics are all equidistributed on $\mqn$:
$$\left(\operatorname{Asc},\operatorname{Dasc}\right),~\left(\operatorname{Plat},\operatorname{Dplat}\right),~\left(\operatorname{Des},\operatorname{Ddes}\right),$$
$$\left(\operatorname{Asc},\operatorname{Uu}\right),~\left(\operatorname{Plat},\operatorname{Pasc}\right),~\left(\operatorname{Des},\operatorname{Dd}\right).$$
So we get the following four identities:
$$\sum_{\sigma\in\mqn}x^{\asc(\sigma)}y^{\dasc(\sigma)}=\sum_{\sigma\in\mqn}x^{\plat(\sigma)}y^{\dplat(\sigma)}=\sum_{\sigma\in\mqn}x^{\des(\sigma)}y^{\ddes(\sigma)},$$
$$\sum_{\sigma\in\mqn}x^{\asc(\sigma)}y^{\dasc(\sigma)}=\sum_{\sigma\in\mqn}x^{\plat(\sigma)}y^{\operatorname{pasc}(\sigma)}=\sum_{\sigma\in\mqn}x^{\des(\sigma)}y^{\ddes(\sigma)},$$
$$\sum_{\sigma\in\mqn}x^{\asc(\sigma)}y^{\operatorname{uu}(\sigma)}=\sum_{\sigma\in\mqn}x^{\plat(\sigma)}y^{\operatorname{pasc}(\sigma)}
=\sum_{\sigma\in\mqn}x^{\des(\sigma)}y^{\operatorname{dd}(\sigma)},$$
$$\sum_{\sigma\in\mqn}x^{\asc(\sigma)}y^{\operatorname{uu}(\sigma)}=\sum_{\sigma\in\mqn}x^{\plat(\sigma)}y^{\dplat(\sigma)}
=\sum_{\sigma\in\mqn}x^{\des(\sigma)}y^{\operatorname{dd}(\sigma)}.$$
\end{theorem}
\begin{proof}
Consider Table~\ref{Table1}. For $C_n\in\operatorname{CQ}_n$, if we switch the $2$-tuples $(a_i,1)$ and $(a_i,2)$ for all $i$ (if any), then we see that
the following bivariable set-valued statistics are equidistributed on $\operatorname{CQ}_n$:
$$\left([n]-\{a_i\mid (a_i,1)\in C_n\},\{a_i\mid (a_i,1)\notin C_n~\&~ (a_i,2)\in C_n\}\right),$$
$$\left([n]-\{a_i\mid (a_i,2)\in C_n\},\{a_i\mid (a_i,1)\in C_n~\&~ (a_i,2)\notin C_n\}\right).$$
which yields $\left(\operatorname{Asc},\operatorname{Dasc}\right)$ and $\left(\operatorname{Plat},\operatorname{Dplat}\right)$ are equidistributed on $\mqn$.

If we switch the $2$-tuples $(a_i,1)$ and $(a_i,3)$ for all $i$ (if any), then we find that
the following bivariable set-valued statistics are equidistributed on $\operatorname{CQ}_n$:
$$\left([n]-\{a_i\mid (a_i,1)\in C_n\},\{a_i\mid (a_i,1)\notin C_n~\&~ (a_i,2)\in C_n\}\right),$$
$$\left([n]-\{a_i\mid (a_i,3)\in C_n\},\{a_i\mid (a_i,2)\in C_n~\&~ (a_i,3)\notin C_n\}\right),$$
which yields $\left(\operatorname{Asc},\operatorname{Dasc}\right)$ and $\left(\operatorname{Des},\operatorname{Ddes}\right)$ are equidistributed on $\mqn$.

If we switch the $2$-tuples $(a_i,1)$ and $(a_i,2)$ for all $i$ (if any), then we see that
the following bivariable set-valued statistics are equidistributed on $\operatorname{CQ}_n$:
$$\left([n]-\{a_i\mid (a_i,1)\in C_n\},\{a_i\mid (a_i,1)\notin C_n~\&~ (a_i,3)\in C_n\}\right),$$
$$\left([n]-\{a_i\mid (a_i,2)\in C_n\},\{a_i\mid (a_i,2)\notin C_n~\&~ (a_i,3)\in C_n\}\right),$$
which yields $\left(\operatorname{Asc},\operatorname{Uu}\right)$ and $\left(\operatorname{Plat},\operatorname{Pasc}\right)$ are equidistributed on $\mqn$.

If we switch the $2$-tuples $(a_i,1)$ and $(a_i,3)$ for all $i$ (if any), then we find that
the following bivariable set-valued statistics are equidistributed on $\operatorname{CQ}_n$:
$$\left([n]-\{a_i\mid (a_i,1)\in C_n\},\{a_i\mid (a_i,1)\notin C_n~\&~ (a_i,3)\in C_n\}\right),$$
$$\left([n]-\{a_i\mid (a_i,3)\in C_n\},\{a_i\mid (a_i,1)\in C_n~\&~ (a_i,3)\notin C_n\}\right),$$
which yields $\left(\operatorname{Asc},\operatorname{Uu}\right)$ and $\left(\operatorname{Des},\operatorname{Dd}\right)$ are equidistributed on $\mqn$.

If we switch the $2$-tuples $(a_i,2)$ and $(a_i,3)$ for all $i$ (if any), then we find that
the following bivariable set-valued statistics are equidistributed on $\operatorname{CQ}_n$:
$$\left([n]-\{a_i\mid (a_i,2)\in C_n\},\{a_i\mid (a_i,1)\in C_n~\&~ (a_i,2)\notin C_n\}\right).$$
$$\left([n]-\{a_i\mid (a_i,3)\in C_n\},\{a_i\mid (a_i,1)\in C_n~\&~ (a_i,3)\notin C_n\}\right),$$
which yields $\left(\operatorname{Plat},\operatorname{Dplat}\right)$ and $\left(\operatorname{Des},\operatorname{Dd}\right)$ are equidistributed on $\mqn$.

In conclusion, the proof is completed by using transitivity.
\end{proof}

\begin{theorem}\label{thm14}
The six bivariable set-valued statistics are equidistributed on $\mqn$:
$$\left(\operatorname{Asc},\operatorname{Lap}\right),~\left(\operatorname{Plat},\operatorname{Lap}\right),~\left(\operatorname{Des},\operatorname{Rpd}\right),$$
$$\left(\operatorname{Asc},\operatorname{Eud}\right),~\left(\operatorname{Plat},\operatorname{Rpd}\right),~\left(\operatorname{Des},\operatorname{Eud}\right).$$
So we get the following six identities:
$$\sum_{\sigma\in\mqn}x^{\asc(\sigma)}y^{\operatorname{lap}(\sigma)}=\sum_{\sigma\in\mqn}x^{\plat(\sigma)}y^{\operatorname{lap}(\sigma)}
=\sum_{\sigma\in\mqn}x^{\des(\sigma)}y^{\operatorname{rpd}(\sigma)},$$
$$\sum_{\sigma\in\mqn}x^{\asc(\sigma)}y^{\operatorname{lap}(\sigma)}=\sum_{\sigma\in\mqn}x^{\plat(\sigma)}y^{\operatorname{rpd}(\sigma)}
=\sum_{\sigma\in\mqn}x^{\des(\sigma)}y^{\operatorname{rpd}(\sigma)},$$
$$\sum_{\sigma\in\mqn}x^{\asc(\sigma)}y^{\operatorname{eud}(\sigma)}=\sum_{\sigma\in\mqn}x^{\plat(\sigma)}y^{\operatorname{rpd}(\sigma)}
=\sum_{\sigma\in\mqn}x^{\des(\sigma)}y^{\operatorname{eud}(\sigma)},$$
$$\sum_{\sigma\in\mqn}x^{\asc(\sigma)}y^{\operatorname{eud}(\sigma)}=\sum_{\sigma\in\mqn}x^{\plat(\sigma)}y^{\operatorname{lap}(\sigma)}
=\sum_{\sigma\in\mqn}x^{\des(\sigma)}y^{\operatorname{eud}(\sigma)},$$
$$\sum_{\sigma\in\mqn}x^{\asc(\sigma)}y^{\operatorname{lap}(\sigma)}=\sum_{\sigma\in\mqn}x^{\plat(\sigma)}y^{\operatorname{rpd}(\sigma)}
=\sum_{\sigma\in\mqn}x^{\des(\sigma)}y^{\operatorname{eud}(\sigma)},$$
$$\sum_{\sigma\in\mqn}x^{\asc(\sigma)}y^{\operatorname{eud}(\sigma)}=\sum_{\sigma\in\mqn}x^{\plat(\sigma)}y^{\operatorname{lap}(\sigma)}
=\sum_{\sigma\in\mqn}x^{\des(\sigma)}y^{\operatorname{rpd}(\sigma)},$$
where the last two identities give the generalizations of~\eqref{Bona} and~\eqref{laprpd} simultaneously.
\end{theorem}
\begin{proof}
Consider Table~\ref{Table1}. For $C_n\in\operatorname{CQ}_n$, if we switch the $2$-tuples $(a_i,1)$ and $(a_i,2)$ for all $i$ (if any), then we see that
the following bivariable set-valued statistics are equidistributed on $\operatorname{CQ}_n$:
$$([n]-\{a_i\mid (a_i,1)\in C_n\},[n]-\{a_i\mid (a_i,1)~\text{or}~ (a_i,2)\in C_n\}),$$
$$([n]-\{a_i\mid (a_i,2)\in C_n\},[n]-\{a_i\mid (a_i,1)~\text{or}~ (a_i,2)\in C_n\}),$$
which yields $\left(\operatorname{Asc},\operatorname{Lap}\right)$ and $\left(\operatorname{Plat},\operatorname{Lap}\right)$ are equidistributed on $\mqn$.

If we switch the $2$-tuples $(a_i,1)$ and $(a_i,3)$ for all $i$ (if any), then we see that
the following bivariable set-valued statistics are equidistributed on $\operatorname{CQ}_n$:
$$([n]-\{a_i\mid (a_i,1)\in C_n\},[n]-\{a_i\mid (a_i,1)~\text{or}~ (a_i,2)\in C_n\}),$$
$$([n]-\{a_i\mid (a_i,3)\in C_n\},[n]-\{a_i\mid (a_i,2)~\text{or}~ (a_i,3)\in C_n\}),$$
which yields $\left(\operatorname{Asc},\operatorname{Lap}\right)$ and $\left(\operatorname{Des},\operatorname{Rpd}\right)$ are equidistributed on $\mqn$.

If we switch the $2$-tuples $(a_i,1)$ and $(a_i,2)$ for all $i$ (if any), then we see that
the following bivariable set-valued statistics are equidistributed on $\operatorname{CQ}_n$:
$$([n]-\{a_i\mid (a_i,1)\in C_n\},[n]-\{a_i\mid (a_i,1)~\text{or}~ (a_i,3)\in C_n\}),$$
$$([n]-\{a_i\mid (a_i,2)\in C_n\},[n]-\{a_i\mid (a_i,2)~\text{or}~ (a_i,3)\in C_n\}),$$
which yields $\left(\operatorname{Asc},\operatorname{Eud}\right)$ and $\left(\operatorname{Plat},\operatorname{Rpd}\right)$ are equidistributed on $\mqn$.

If we switch the $2$-tuples $(a_i,1)$ and $(a_i,3)$ for all $i$ (if any), then we see that
the following bivariable set-valued statistics are equidistributed on $\operatorname{CQ}_n$:
$$([n]-\{a_i\mid (a_i,1)\in C_n\},[n]-\{a_i\mid (a_i,1)~\text{or}~ (a_i,3)\in C_n\}),$$
$$([n]-\{a_i\mid (a_i,3)\in C_n\},[n]-\{a_i\mid (a_i,1)~\text{or}~ (a_i,3)\in C_n\}),$$
which yields $\left(\operatorname{Asc},\operatorname{Eud}\right)$ and $\left(\operatorname{Des},\operatorname{Eud}\right)$ are equidistributed on $\mqn$.

If we switch the $2$-tuples $(a_i,2)$ and $(a_i,3)$ for all $i$ (if any), then we see that
the following bivariable set-valued statistics are equidistributed on $\operatorname{CQ}_n$:
$$([n]-\{a_i\mid (a_i,1)\in C_n\},[n]-\{a_i\mid (a_i,1)~\text{or}~ (a_i,2)\in C_n\}),$$
$$([n]-\{a_i\mid (a_i,1)\in C_n\},[n]-\{a_i\mid (a_i,1)~\text{or}~ (a_i,3)\in C_n\}),$$
which yields $\left(\operatorname{Asc},\operatorname{Lap}\right)$ and $\left(\operatorname{Asc},\operatorname{Eud}\right)$ are equidistributed on $\mqn$.

In conclusion, the proof is completed by using transitivity.
\end{proof}

We say that a joint distribution of (set-valued) statistics or a multivariate polynomial is {\it symmetric} if
it is invariant under any permutation of its indeterminates.
We can now present the following two results, and the proofs follow in the same way as the proof of Theorem~\ref{thm14}.
\begin{theorem}\label{Thmset01}
The six set-valued statistics are all equidistributed on $\mqn$:
$$\operatorname{Dasc},\operatorname{Dplat},\operatorname{Ddes},\operatorname{Pasc},\operatorname{Uu},~\operatorname{Dd}.$$
Moreover,
if one select any two set-valued statistics from these six set-valued statistics, then the selected two set-valued statistics
are symmetric on $\mqn$.
\end{theorem}

\begin{theorem}\label{Thmset02}
The following triple set-valued statistics are all symmetric on $\mqn$:
$$(\operatorname{Asc}(\sigma),\operatorname{Plat}(\sigma),\operatorname{Des}(\sigma)),~(\operatorname{Lap}(\sigma),\operatorname{Rpd}(\sigma),\operatorname{Eud}(\sigma)),$$
$$(\operatorname{Dasc}(\sigma),\operatorname{Pasc}(\sigma),\operatorname{Dd}(\sigma)),~(\operatorname{Ddes}(\sigma),\operatorname{Dplat}(\sigma),\operatorname{Uu}(\sigma)).$$
\end{theorem}
We now give an example to illustrate the proof of the symmetric of $(\operatorname{Ddes},\operatorname{Pasc})$.
\begin{example}
Let $\sigma$ and $C_7$ be the given in Example~\ref{example10}. Then we have $\operatorname{Ddes}(\sigma)=\{5\}$ and
$\operatorname{Pasc}(\sigma)=\{2\}$.
Let $\Phi$ be the bijection on $\operatorname{CQ}_7$ that is defined by
$$(a_i,2)\leftrightarrow (a_i,3),~{\text {where $1\leqslant i\leqslant 6$}}.$$ In other words, we just switch the $2$-tuples $(a_i,2)$ and $(a_i,3)$ for all $i$ (if any).
Thus $$\Phi\left((0,0)(1,2)(2,3)(1,1)(1,3)(5,2)(4,1)\right)=(0,0)(1,3)(2,2)(1,1)(1,2)(5,3)(4,1).$$
It is routine to verify that $\phi^{-1}\left(\Phi(C_7)\right)=774415\textbf{5}661233\textbf{2}$.
Therefore, we have
$$\operatorname{Ddes}\left(\phi^{-1}\left(\Phi(C_7)\right)\right)=\{2\},~\operatorname{Pasc}\left(\phi^{-1}\left(\Phi(C_7)\right)\right)=\{5\}.$$
\end{example}
\begin{corollary}
For $n\geqslant 1$, both of the following two polynomials are symmetric in their variables:
$$C_n(x,y,z)=\sum_{\sigma\in\mqn}x^{\asc(\sigma)}y^{\plat(\sigma)}z^{\des(\sigma)},$$
$$N_n(x,y,z)=\sum_{\sigma\in\mqn}x^{\lap\sigma)}y^{\rpd(\sigma)}z^{\eudd(\sigma)}.$$
\end{corollary}
As discussed in introduction, the symmetry of $C_n(x,y,z)$ has been extensively studied, see~\cite{Chen21,Haglund12} and references therein.
In Section~\ref{section5}, we shall show the $e$-positivity of $N_n(x,y,z)$.
\section{Bijections among $\operatorname{SP}$-codes, trapezoidal words and perfect matchings}\label{section402}
Following Riordan~\cite{Riordan76}, we say that a word $t=t_1t_2\cdots t_n$ is a {\it Riordan trapezoidal word} if the element $t_i$ takes the values $1,2,\ldots,2i-1$.
Let $\operatorname{RT}_n$ be the set of Riordan trapezoidal words of length $n$. In particular, $$\operatorname{RT}_1=\{1\},~\operatorname{RT}_2=\{11,12,13\}.$$
Besides~\eqref{QnTn}, Dumont~\cite{Dumont80} gave another two interpretations of $C_n(x,y,z)$ in terms of Dumont trapezoidal words and perfect matchings.
The Dumont trapezoidal word~\cite{Dumont80} is a variant Riordan trapezoidal word.
We say that a word $w=w_1w_2\cdots w_n$ is a {\it Dumont trapezoidal word} of length $n$ if $0\leqslant |w_i|<i$, where $w_i$ are all integers.
Let $\operatorname{DT}_n$ denote the set of Dumont trapezoidal words of length $n$. For convenience, we set $\overline{i}=-i$.
In particular, $$\operatorname{DT}_1=\{0\},~\operatorname{DT}_2=\{00,01,0\overline{1}\}.$$
Given $w\in\operatorname{DT}_n$. Let $\operatorname{dist}(w)$ be the number of distinct elements in $w$, and we define
\begin{align*}
\operatorname{nneg}(\sigma)&=n-\{w_i\mid w_i<0\},~
\operatorname{npos}(\sigma)=n-\{w_i\mid w_i>0\}.
\end{align*}
Dumont~\cite[Section~2.3]{Dumont80} found that
$$C_n(x,y,z)=\sum_{w\in\in\operatorname{DT}_n}x^{\operatorname{dist}(w)}y^{\operatorname{nneg}(\sigma)}z^{\operatorname{npos}(\sigma)}.$$

A \emph{perfect matching} of $[2n]$ is a set partition of $[2n]$ with blocks (disjoint nonempty subsets) of size exactly 2.
Let $\mmn$ be the set of perfect matchings of $[2n]$, and let $\m\in\mmn$.
The \emph{standard form} of $\m$ is a list of blocks $$\{(i_1,j_1),(i_2,j_2),\ldots,(i_n,j_n)\}$$ such that
$i_r<j_r$ for all $1\leqslant r\leqslant n$ and $1=i_1<i_2<\cdots <i_n$.
As usual, we always write $\m$ in standard form.
It is well known that $\m$ can
be regarded as a fixed-point-free involution on $[2n]$.
In particular, $$\mm_2=\{(1,2)\},~\mm_4=\{(1,2)(3,4),(1,3)(2,4),(1,4)(2,3)\}.$$

Motivated by Theorem~\ref{thmCQ}, it is natural to explore the bijections among $\operatorname{SP}$-codes, trapezoidal words and perfect matchings.
\begin{theorem}
For $n\geqslant 1$, we have
\begin{equation}
\operatorname{CQ}_n\cong\operatorname{DT}_n\cong\operatorname{RT}_n\cong\mmn.
\end{equation}
\end{theorem}
\begin{proof}
\quad $(i)$
Let $\varphi_1: \operatorname{DT}_n\rightarrow\operatorname{RT}_n$
be the bijection defined by
\begin{equation}
\varphi_1(w_i)=\left\{
               \begin{array}{ll}
                 1, & \hbox{if $w_i=0$;} \\
                 2k, & \hbox{if $w_i=k>0$;} \\
                 2k+1, & \hbox{if $w_i=\overline{k}<0$,}
               \end{array}
             \right.
\end{equation}
which yields that $\operatorname{DT}_n\cong\operatorname{RT}_n$.

\quad $(ii)$
Now we start to construct a bijection, denoted by $\varphi_2$, from $\operatorname{RT}_n$ to $\mmn$.
When $n=1$, we have $\operatorname{RT}_n=\{1\}$. Set $\varphi_2(\textbf{1})=(\textbf{1},2)$. When $n=2$, we have
$\operatorname{RT}_2=\{11,12,13\}$. We set
$$\varphi_2(1\textbf{1})=(\textbf{1},4)(2,3),~\varphi_2(1\textbf{2})=(\textbf{2},4)(1,3),~
\varphi_2(1\textbf{3})=(\textbf{3},4)(1,2).$$ We proceed by induction.
Let $n=m$. Suppose that $\varphi_2$ is a bijection from $\operatorname{RT}_m$ to $\mm_{2m}$.
Given $\m=(i_1,j_1)(i_2,j_2)\cdots (i_m,j_m)\in \mm_{2m}$. Suppose that $\varphi_2(t)=\m$, where $t=t_1t_2\cdots t_m\in \operatorname{RT}_m$.
For $1\leqslant i\leqslant 2m+1$, we let $\varphi_2$ be the following algorithm, which can be used to generate perfect matchings in $\mm_{2m+2}$:
 \begin{itemize}
   \item $\varphi_2(t_1t_2\cdots t_m i)=(i,2m+2)(i_1',j_1')(i_2',j_2')\cdots (i_m',j_m')$, where $(i_1',j_1')(i_2',j_2')\cdots (i_m',j_m')$ is a perfect matching
of the elements in $[2m+2]-\{i,2m+2\}$ such that the elements in $(i_1',j_1')\cdots (i_m',j_m')$ keeps the same order relationships they have in
$(i_1,j_1)\cdots (i_m,j_m)$.
 \end{itemize}
Clearly, $\varphi_2$ is the desired bijection. See Example~\ref{ExMat} for an illustration.

\quad $(iii)$
Now we start to construct a bijection, denoted by $\varphi_3$, from $\operatorname{DT}_n$ to $\operatorname{CQ}_n$.
When $n=1$, we set $\varphi_3(0)=(0,0)$.
When $n=m$, suppose $\varphi_3$ is a bijection from $\operatorname{DT}_m$ to $\operatorname{CQ}_m$.
Consider the case $n=m+1$. Let $w=w_1w_2\cdots w_{m+1}\in \operatorname{DT}_{m+1}$. Then $w'=w_1w_2\cdots w_{m}\in \operatorname{DT}_{m}$ and
$\varphi_3(w')=((0,0),(a_1,b_1),(a_2,b_2)\ldots,(a_{m-1},b_{m-1}))\in \operatorname{CQ}_m$.
We distinguish three cases:
\begin{enumerate}
\item [$(c_1)$] $w_{m+1}=k$ and $k\in \{w_1,w_2,\ldots,w_m\}$ if and only if $$(a_m,b_m)=(j,1),~{\text{where}}~j=\max\{i\mid w_i=k,~1\leqslant i\leqslant m\};$$
\item [$(c_2)$] $w_{m+1}=-j$ and $j\notin \{w_1,w_2,\ldots,w_m\}$ if and only if $(a_m,b_m)=(j,2)$, where $1\leqslant j\leqslant m$;
\item [$(c_3)$] $w_{m+1}=j$ and $j\notin \{w_1,w_2,\ldots,w_m\}$ if and only if $(a_m,b_m)=(j,3)$, where $1\leqslant j\leqslant m$.
 \end{enumerate}
It is routine to check that $\varphi_3$ is the desired bijection. In particular, when $n=2,3$, we have
\begin{align*}
\varphi_3(00)&=(0,0)(1,1),~
\varphi_3(0\overline{1})=(0,0)(1,2),~
\varphi_3(01)=(0,0)(1,3);\\
\varphi_3(000)&=(0,0)(1,1)(2,1),~\varphi_3(00\overline{1})=(0,0)(1,1)(1,2),~\varphi_3(00{1})=(0,0)(1,1)(1,3),\\
\varphi_3(00\overline{2})&=(0,0)(1,1)(2,2),~\varphi_3(002)=(0,0)(1,1)(2,3),~\varphi_3(0\overline{1}0)=(0,0)(1,2)(1,1),\\
\varphi_3(0\overline{1}1)&=(0,0)(1,2)(1,3),~\varphi_3(0\overline{1}~\overline{1})=(0,0)(1,2)(2,1),~\varphi_3(0\overline{1}~\overline{2})=(0,0)(1,2)(2,2),\\
\varphi_3(0\overline{1}2)&=(0,0)(1,2)(2,3),~\varphi_3(010)=(0,0)(1,3)(1,1),~\varphi_3(01\overline{1})=(0,0)(1,3)(1,2),\\
\varphi_3(011)&=(0,0)(1,3)(2,1),~\varphi_3(01\overline{2})=(0,0)(1,3)(2,2),~\varphi_3(012)=(0,0)(1,3)(2,3).
\end{align*}
This completes the proof.
\end{proof}

As illustrations of $\varphi_2(w)$ and $\varphi_3(w)$, we give another two examples.
\begin{example}\label{ExMat}
Given $t=1\operatorname{-}1\operatorname{-}1\operatorname{-}3\operatorname{-}2\operatorname{-}10$.
We give the procedure of creating $\varphi_2(t)$.
\begin{align*}
\textbf{1}&\Leftrightarrow (1,2),\\
1\operatorname{-}\textbf{1}&\Leftrightarrow (\textbf{1},4)(2,3),\\
1\operatorname{-}{1}\operatorname{-}\textbf{1}&\Leftrightarrow (\textbf{1},6)(2,5)(3,4),\\
1\operatorname{-}{1}\operatorname{-}{1}\operatorname{-}\textbf{3}&\Leftrightarrow (\textbf{3},8)(1,7)(2,6)(4,5),\\
1\operatorname{-}{1}\operatorname{-}{1}\operatorname{-}{3}\operatorname{-}{\textbf{2}}&\Leftrightarrow (\textbf{2},10)(4,9)(1,8)(3,7)(5,6),\\
1\operatorname{-}1\operatorname{-}1\operatorname{-}3\operatorname{-}2\operatorname{-}\textbf{10}&\Leftrightarrow (\textbf{10},12)(2,11)(4,9)(1,8)(3,7)(5,6).
\end{align*}
Thus $\varphi_2(t)=({10},12)(2,11)(4,9)(1,8)(3,7)(5,6)$. Conversely, we get $\varphi_2^{-1}(\varphi_2(t))=t$.
\end{example}

\begin{example}\label{ExDT}
Given $w=0\operatorname{-}0\operatorname{-}0\operatorname{-}\overline{1}\operatorname{-}1\operatorname{-}5\operatorname{-}1$.
We give the procedure of creating $\varphi_3(w)$.
\begin{align*}
\textbf{0}&\Leftrightarrow (0,0),\\
0\operatorname{-}\textbf{0}&\Leftrightarrow (0,0)(\textbf{1},\textbf{1}),\\
0\operatorname{-}{0}\operatorname{-}\textbf{0}&\Leftrightarrow (0,0)(1,1)(\textbf{2},\textbf{1}),\\
0\operatorname{-}0\operatorname{-}0\operatorname{-}\overline{\textbf{1}}&\Leftrightarrow (0,0)(1,1)({2},{1})(\textbf{1},\textbf{2}),\\
0\operatorname{-}0\operatorname{-}0\operatorname{-}\overline{1}\operatorname{-}\textbf{1}&\Leftrightarrow (0,0)(1,1)({2},{1})({1},2)(\textbf{1},\textbf{3}),\\
0\operatorname{-}0\operatorname{-}0\operatorname{-}\overline{1}\operatorname{-}1\operatorname{-}5&\Leftrightarrow(0,0)(1,1)({2},{1})({1},2)({1},{3})(\textbf{5},\textbf{3}),\\
0\operatorname{-}0\operatorname{-}0\operatorname{-}\overline{1}\operatorname{-}1\operatorname{-}5\operatorname{-}\textbf{1}
&\Leftrightarrow(0,0)(1,1)({2},{1})({1},2)({1},{3})({5},{3})(\textbf{5},\textbf{1}).
\end{align*}
Thus $\varphi_3(w)=(0,0)(1,1)({2},{1})({1},2)({1},{3})({5},{3})({5},{1})$. Conversely, we get $\varphi_3^{-1}(\varphi_3(w))=w$.
\end{example}
\section{The $e$-positivity of the enumerators by $(\lap,\eudd,\rpd)$}\label{section5}
\subsection{Preliminary}
\hspace*{\parindent}

Let $\operatorname{X}_n=\{x_1,x_2,\ldots,x_n\}$ be a set of commuting variables.
Define $$S_n(x)=\prod_{i=1}^n(x-x_i)=\sum_{k=0}^n{(-1)}^ke_{k}x^{n-k}.$$
Then the {\it $k$-th
elementary symmetric function} associated with $\operatorname{X}_n$ is defined by $$e_k=\sum_{1\leqslant i_1<i_2<\cdots<i_k\leqslant n}x_{i_1}x_{i_2}\cdots x_{i_k}.$$
In particular, $$e_0=1,~e_1=\sum_{i=1}^nx_i,~e_n=x_1x_2\cdots x_n.$$
A function $f(x_1,x_2,\ldots)\in \mathbb{R}[x_1,x_2,\ldots]$ is said to be {\it symmetric} if it is invariant under
any permutation of its indeterminates.
We say that a symmetric function is {\it $e$-positive} if it can be written as a nonnegative linear combination of elementary symmetric
functions.

For an alphabet $A$, let $\mathbb{Q}[[A]]$ be the rational commutative ring of formal power
series in monomials formed from letters in $A$. Following Chen~\cite{Chen93}, a {\it context-free grammar} over
$A$ is a function $G: A\rightarrow \mathbb{Q}[[A]]$ that replaces each letter in $A$ by a formal function over $A$.
The formal derivative $D_G$ with respect to $G$ satisfies the derivation rules:
$$D_G(u+v)=D_G(u)+D_G(v),~D_G(uv)=D_G(u)v+uD_G(v).$$
So the Leibniz rule holds:
$$D_G^n(uv)=\sum_{k=0}^n\binom{n}{k}D_G^k(u)D_G^{n-k}(v).$$
See~\cite{Dumont96,Ma1902} for some examples of context-free grammars.

Recently, two methods are developed in the theory of context-free grammars, i.e., grammatical labeling and the change of grammars.
A {\it grammatical labeling} is an assignment of the underlying elements of a combinatorial structure
with variables, which is consistent with the substitution rules of a grammar (see~\cite{Chen17} for details).
A {\it change of grammars} is a substitution method in which the original grammars are replaced with functions of other grammars.
In particular, the following type of change of grammars can be used to study the $\gamma$-positivity and partial $\gamma$-positivity of several enumerative polynomials (see~\cite{Chen21,Ma1902,Ma2021} for details):
\begin{equation*}\label{change-grammars01}
\left\{
  \begin{array}{ll}
    u=xy, &  \\
    v=x+y. &
  \end{array}
\right.
\end{equation*}

Let $G$ be the following grammar
\begin{equation}\label{xyz}
G=\{x \rightarrow xyz, y\rightarrow xyz, z\rightarrow xyz\}.
\end{equation}
It has been shown by Dumont~\cite{Dumont80}, Haglund-Visontai~\cite{Haglund12} and Chen etal.~\cite{Chen2102} that
$$D_G^n(x)=C_n(x,y,z).$$

Very recently, Chen-Fu~\cite{Chen21} introduced a new type of change of grammars:
\begin{equation}\label{change-grammars02}
\left\{
  \begin{array}{ll}
    u=x+y+z, &  \\
    v=xy+yz+zx, &\\
    w=xyz.
  \end{array}
\right.
\end{equation}
Combining~\eqref{xyz} and~\eqref{change-grammars02},
one can easily verify that
$D_{G}(u)=3w,~D_{{G}}(v)=2uw,~D_{G}(w)=vw$, which yield a new grammar
\begin{equation}\label{G8def}
H=\{u\rightarrow 3w, v\rightarrow 2uw,~w\rightarrow vw\}.
\end{equation}
For any $n\geqslant 1$, Chen-Fu~\cite[]{Chen21} discovered that
\begin{equation}\label{Chen22}
C_n(x,y,z)=D_G^n(x)=D_H^{n-1}(w)=\sum_{i+2j+3k=2n+1}\gamma_{n,i,j,k}u^iv^jw^k,
\end{equation}
which leads to the $e$-positive expansion~\eqref{Cnxyz}.

We can now present the main result of this section.
\begin{theorem}\label{mainthm03}
Let
\begin{equation}\label{Nnxyz01}
N_n(x,y,z)=\sum_{\sigma\in\mqn}x^{\lap(\sigma)}y^{\eudd(\sigma)}z^{\rpd(\sigma)}.
\end{equation}
Then we have
\begin{equation}\label{Nnxyz02}
\begin{split}
N_n(x,y,z)&=\sum_{i+2j+3k=2n+1}3^i\gamma_{n,i,j,k}(x+y+z)^{j}(xyz)^k,
\end{split}
\end{equation}
where $\gamma_{n,i,j,k}$ is the same as in~\eqref{Chen22}, i.e., $\gamma_{n,i,j,k}$ equals the number of 0-1-2-3 increasing plane trees
on $[n]$ with $k$ leaves, $j$ degree one vertices and $i$ degree two vertices.
\end{theorem}

Throughout this section, we always let $$w_1=x+y+z,~w_2=xy+yz+zx,~w_3=xyz.$$
Below are the polynomials $N_n(x,y,z)$ for $n\leqslant 6$:
\begin{align*}
N_1(x,y,z)&=w_3,\\
N_2(x,y,z)&=w_1w_3,\\
N_{3}(x,y,z)&=w_1^2w_3+6w_3^2,\\
N_{4}(x,y,z)&=w_1^3w_3+24w_1w_3^2+6w_3^3,\\
N_{5}(x,y,z)&=w_1^4w_3+66w_1^2w_3^2+42w_1w_3^3+144 w_3^3,\\
N_{6}(x,y,z)&= w_1^5w_3+156w_1^3w_3^2+192w_1^2w_3^3+1224w_1w_3^3+540 w_3^4.
\end{align*}

\begin{example}
For the elements in $\mq_2$, we have
\begin{equation*}
\begin{split}
\lap(1122)&=\lap(011220)=2,~\eudd(1122)=\eudd(011220)=1,~\rpd(1122)=\rpd(011220)=1,\\
\lap(1221)&=\lap(012210)=1,~\eudd(1221)=\eudd(012210)=2,~\rpd(1221)=\rpd(012210)=1,\\
\lap(2211)&=\lap(022110)=1,~\eudd(2211)=\eudd(022110)=1,~\rpd(2211)=\rpd(022110)=2.
\end{split}
\end{equation*}
Thus $N_2(x,y,z)=xyz(x+y+z)=w_1w_3$. See Figure~\ref{Fig003} for an illustration.
\end{example}
\subsection{Proof of Theorem~\ref{mainthm03}}
\hspace*{\parindent}

As discussed in Section~\ref{section4}, we shall use simplified ternary increasing trees for convenience.
See Figure~\ref{Fig003} for an illustration, where the left figure represents the three different figures in the right.
The weights of $\sigma\in\mqn$ and $C_n\in\operatorname{CQ}_n$ are respectively defined as follows:
$$E_1(\sigma)=x^{\lap(\sigma)y^{\eudd(\sigma)}}z^{\rpd(\sigma)},$$
\begin{equation}\label{Weight}
E_2(C_n)=x^{n-\#\{a_i\mid (a_i,1)~\text{or}~ (a_i,2)\in C_n\}}y^{n-\#\{a_i\mid (a_i,1)~\text{or}~ (a_i,3)\in C_n\}}z^{n-\#\{a_i\mid (a_i,2)~\text{or}~ (a_i,3)\in C_n\}}.
\end{equation}
From Table~\ref{Table1}, we see that $E_1(\sigma)=E_2(C_n)$, where $C_n$ is the corresponding $\operatorname{SP}$-code of $\sigma$.
When $n=1$, the $\operatorname{SP}$-code $(0,0)$ corresponds to the Stirling permutation $11$. Clearly,
$$E_1(11)=E_2((0,0))=xyz=w_3.$$
When $n=2$, the weights of elements in $\mq_2$ and $\operatorname{CQ}_n$ are given as follows:
\begin{equation*}\label{w3}
{\underbrace{2211\leftrightarrow (0,0)(1,1)}_{xyz^2=w_3z},~\underbrace{1221\leftrightarrow (0,0)(1,2)}_{xy^2z=w_3y},~\underbrace{1122\leftrightarrow (0,0)(1,3)}_{x^2yz=w_3x}},
\end{equation*}
and the sum of weights is given by $w_3(x+y+z)=w_3w_1$.

\begin{figure}
\begin{center}
\begin{tikzpicture}
[emph/.style={edge from parent/.style={snakeline,draw}}]
\node (1) [circle,draw] {1}
    child [emph] {node (2) [circle,draw] {2}};
\path (1.east) node[above right] {[$p_1$]};
\path (2.east) node[above right] {[$p_3$]};
\end{tikzpicture}
=
\begin{tikzpicture}
[level 1/.style = {sibling distance = .7cm},
NONE/.style={edge from parent/.style={draw=none}}]
\node [circle,draw] {1}
    child {node [circle,draw] {2}}
    child [NONE] {}
    child [NONE] {};
\end{tikzpicture}
or
\begin{tikzpicture}
[level 1/.style = {sibling distance = .7cm},
NONE/.style={edge from parent/.style={draw=none}}]
\node [circle,draw] {1}
    child [NONE] {}
    child {node [circle,draw] {2}}
    child [NONE] {};
\end{tikzpicture}
or
\begin{tikzpicture}
[level 1/.style = {sibling distance = .7cm},
NONE/.style={edge from parent/.style={draw=none}}]
\node [circle,draw] {1}
    child [NONE] {}
    child [NONE] {}
    child {node [circle,draw] {2}};
\end{tikzpicture}
\caption{$N_2(x,y,z)=(x+y+z)xyz$.}
\label{Fig003}
\end{center}
\end{figure}
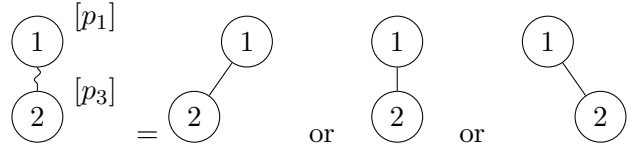

Given $C_n=(0,0)(a_1,b_1)(a_2,b_2)\cdots (a_{n-1},b_{n-1})\in \operatorname{CQ}_n$.
Consider the elements in $\operatorname{CQ}_{n+1}$ generated from $C_n$ by appending the $2$-tuples $(a_n,b_n)$, where $1\leqslant a_n\leqslant n$ and $1\leqslant b_n\leqslant 3$.
Let $T$ be the corresponding ternary increasing tree of $C_n$.
We can add $n+1$ to $T$ as a child of a vertex, which is not of degree three. Let $T'$ be the resulting ternary increasing tree.
We first give a labeling of $T$ as follows. Label a leaf by $p_3$, a degree
one vertex by $p_1$, a degree two vertex by $p_2$ and a degree three vertex by $1$.

\begin{figure}[ht]
\begin{center}
\begin{tikzpicture}
[emph/.style={edge from parent/.style={snakeline,draw}}]
\node (1) [circle,draw] {1}
    child [emph] {node (2) [circle,draw] {2}
    				child {node (3) [circle,draw] {3}}
    };
\path (1.east) node[above right] {[$p_1$]};
\path (2.east) node[above right] {[$p_1$]};
\path (3.east) node[above right] {[$p_3$]};
\end{tikzpicture}
+
\begin{tikzpicture}
[level 1/.style = {sibling distance = .7cm},
NONE/.style={edge from parent/.style={draw=none}}]
\node (1) [circle,draw] {1}
    child {node (2) [circle,draw] {2}}
    child {node (3) [circle,draw] {3}}
    child [NONE] {};
\path (1.east) node[above right] {[$p_2$]};
\path (2.west) node[above left] {[$p_3$]};
\path (3.east) node[above right] {[$p_3$]};
\end{tikzpicture}
+
\begin{tikzpicture}
[level 1/.style = {sibling distance = .7cm},
NONE/.style={edge from parent/.style={draw=none}}]
\node (1) [circle,draw] {1}
    child {node (2) [circle,draw] {3}}
    child {node (3) [circle,draw] {2}}
    child [NONE] {};
\path (1.east) node[above right] {[$p_2$]};
\path (2.west) node[above left] {[$p_3$]};
\path (3.east) node[above right] {[$p_3$]};
\end{tikzpicture}
+
\begin{tikzpicture}
[level 1/.style = {sibling distance = .7cm},
NONE/.style={edge from parent/.style={draw=none}}]
\node (1) [circle,draw] {1}
    child {node (2) [circle,draw] {2}}
    child [NONE] {}
    child {node (3) [circle,draw] {3}};
\path (1.east) node[above right] {[$p_2$]};
\path (2.west) node[above left] {[$p_3$]};
\path (3.east) node[above right] {[$p_3$]};
\end{tikzpicture}
+
\begin{tikzpicture}
[level 1/.style = {sibling distance = .7cm},
NONE/.style={edge from parent/.style={draw=none}}]
\node (1) [circle,draw] {1}
    child {node (2) [circle,draw] {3}}
    child [NONE] {}
    child {node (3) [circle,draw] {2}};
\path (1.east) node[above right] {[$p_2$]};
\path (2.west) node[above left] {[$p_3$]};
\path (3.east) node[above right] {[$p_3$]};
\end{tikzpicture}
+
\begin{tikzpicture}
[level 1/.style = {sibling distance = .7cm},
NONE/.style={edge from parent/.style={draw=none}}]
\node (1) [circle,draw] {1}
    child [NONE] {}
    child {node (2) [circle,draw] {2}}
    child {node (3) [circle,draw] {3}};
\path (1.east) node[above right] {[$p_2$]};
\path (2.west) node[above left] {[$p_3$]};
\path (3.east) node[above right] {[$p_3$]};
\end{tikzpicture}
+
\begin{tikzpicture}
[level 1/.style = {sibling distance = .7cm},
NONE/.style={edge from parent/.style={draw=none}}]
\node (1) [circle,draw] {1}
    child [NONE] {}
    child {node (2) [circle,draw] {3}}
    child {node (3) [circle,draw] {2}};
\path (1.east) node[above right] {[$p_2$]};
\path (2.west) node[above left] {[$p_3$]};
\path (3.east) node[above right] {[$p_3$]};
\end{tikzpicture}
\caption{$N_3(x,y,z)=(x+y+z)^2xyz+6(xyz)^2$}
\label{N3}
\end{center}
\end{figure}
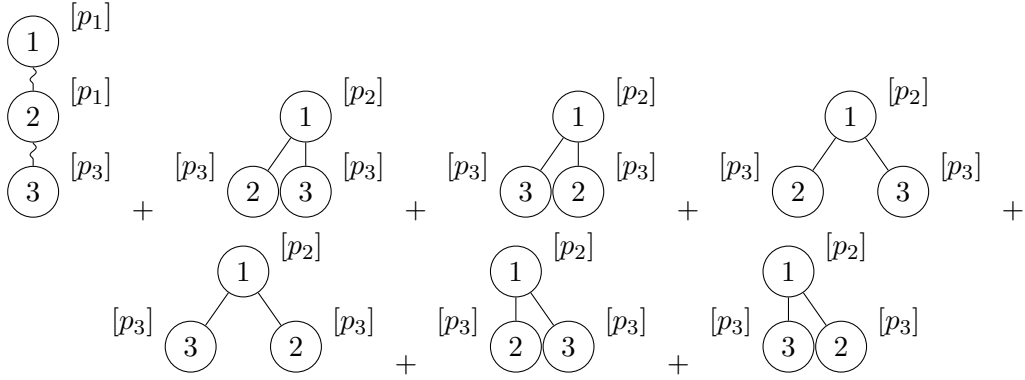
\begin{figure}[ht]
\begin{center}
\begin{tikzpicture}
[level distance = 1 cm,
NONE/.style={edge from parent/.style={draw=none}}]
\node (1) [circle,draw] {}
    child [NONE] {node [circle,draw=none] {}};
\path (1.east) node[above right] {[$p_3$]};
\end{tikzpicture}
, ~
\begin{tikzpicture}
[level distance = 1 cm,
emph/.style={edge from parent/.style={snakeline,draw}}]
\node (1) [circle,draw] {}
    child [emph] {node [circle,draw] {}};
\path (1.east) node[above right] {[$p_1$]};
\end{tikzpicture}
=
\begin{tikzpicture}
[level distance = 1 cm,
level 1/.style = {sibling distance = .7cm},
NONE/.style={edge from parent/.style={draw=none}}]
\node [circle,draw] {}
    child {node [circle,draw] {}}
    child [NONE] {}
    child [NONE] {};
\end{tikzpicture}
or
\begin{tikzpicture}
[level distance = 1 cm,
level 1/.style = {sibling distance = .7cm},
NONE/.style={edge from parent/.style={draw=none}}]
\node [circle,draw] {}
    child [NONE] {}
    child {node [circle,draw] {}}
    child [NONE] {};
\end{tikzpicture}
or
\begin{tikzpicture}
[level distance = 1 cm,
level 1/.style = {sibling distance = .7cm},
NONE/.style={edge from parent/.style={draw=none}}]
\node [circle,draw] {}
    child [NONE] {}
    child [NONE] {}
    child {node [circle,draw] {}};
\end{tikzpicture}
\\

\begin{tikzpicture}
[level distance = 1 cm,
level 1/.style = {sibling distance = .7cm},
NONE/.style={edge from parent/.style={draw=none}}]
\node (1) [circle,draw] {}
    child {node [circle,draw] {}}
    child {node [circle,draw] {}}
    child [NONE] {};
\path (1.east) node[above right] {[$p_2$]};
\end{tikzpicture}
, ~
\begin{tikzpicture}
[level distance = 1 cm,
level 1/.style = {sibling distance = .7cm},
NONE/.style={edge from parent/.style={draw=none}}]
\node (1) [circle,draw] {}
    child [NONE] {}
    child {node [circle,draw] {}}
    child {node [circle,draw] {}};
\path (1.east) node[above right] {[$p_2$]};
\end{tikzpicture}
, ~
\begin{tikzpicture}
[level distance = 1 cm,
level 1/.style = {sibling distance = .7cm},
NONE/.style={edge from parent/.style={draw=none}}]
\node (1) [circle,draw] {}
    child {node [circle,draw] {}}
    child [NONE] {}
    child {node [circle,draw] {}};
\path (1.east) node[above right] {[$p_2$]};
\end{tikzpicture}
, ~
\begin{tikzpicture}
[level distance = 1 cm,
level 1/.style = {sibling distance = .7cm}]
\node (1) [circle,draw] {}
    child {node [circle,draw] {}}
    child {node [circle,draw] {}}
    child {node [circle,draw] {}};
\path (1.east) node[above right] {[$1$]};
\end{tikzpicture}
\end{center}
\caption{Labeling schemes of simplified ternary increasing trees.}
\label{Fig005}
\end{figure}
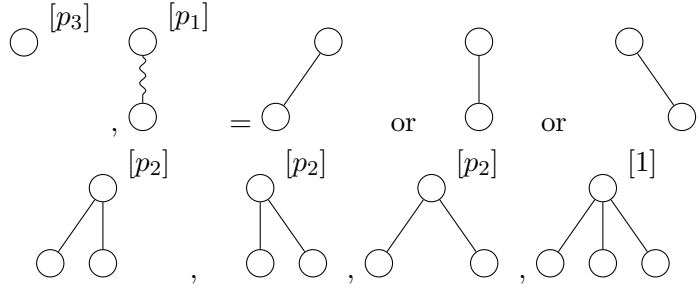

The $2$-tuples $(a_n,b_n)$ can be divided into three classes:
\begin{itemize}
  \item if $a_n\neq a_i$ for all $1\leqslant i\leqslant n-1$, then we must add $n+1$ to a leaf of $T$.
 This operation corresponds to the change of weights
  \begin{equation}\label{G1}
 E_2(C_n)\rightarrow E_2(C_{n+1})=E_2(C_n)(x+y+z),
  \end{equation}
which yields the substitution $p_3\rightarrow p_1p_3$, see Figure~\ref{Fig003} and the first case in Figure~\ref{N3} for illustrations.
Thus the contribution of any leaf to the weight is $xyz$ and that of a degree one vertex is $x+y+z$ (represents this vertex may has a left, middle or right child).
When we compute the corresponding enumerative polynomials of Stirling permutations, it follows from~\eqref{G1} that we need to set
\begin{equation}\label{G2}
p_1=x+y+z,~p_3=xyz;
\end{equation}
  \item if there is exactly one 2-tuple $(a_i,b_i)$ in $C_n$ such that $a_n=a_i$, then we must add $n+1$ to $T$ as a child of the node $a_i$.
  Note that the node $a_i$ already has one child $i+1$, and $n+1$ becomes the second child of $a_i$. There are six cases to add $n+1$.
As illustrations, the last six cases in Figure~\ref{N3} are the total possibilities when we add $3$ to the simplified ternary increasing trees in Figure~\ref{Fig003} as the second child of the node $1$.
This operation corresponds to the substitution $p_1\rightarrow 6p_2p_3$.
From~\eqref{Weight}, we see that each degree two vertex makes no contribution to the weight. Thus we need to set $p_2=1$ when we compute the corresponding enumerative polynomial of the joint distribution of $(\lap,\eudd,\rpd)$;
  \item if there are exactly two 2-tuples $(a_i,b_i)$ and $(a_j,b_j)$ in $C_n$ such that $a_n=a_i=a_j$ and $i<j$, then we must add $n+1$ to $T$ as the third child of $a_i$, and $n+1$ becomes a leaf with label $p_3$. This operation corresponds to the substitution $p_2\rightarrow p_3$. From~\eqref{Weight},
   we see that each degree three vertex makes
no contribution to the weight, and so we label all degree three vertices by $1$.
\end{itemize}
The aforementioned three cases exhaust all the possibilities to construct a $\operatorname{SP}$-code of length $n+1$
from a $\operatorname{SP}$-code of length $n$ by appending $2$-tuples $(a_n,b_n)$. In conclusion, each case corresponds to an application of a
substitution rule in the following grammar:
\begin{equation}\label{Iw3}
I=\{p_3\rightarrow p_1p_3,~p_1\rightarrow 6p_2p_3,~p_2\rightarrow p_3\},
\end{equation}
and the corresponding labeling schemes are illustrated in Figure~\ref{Fig005}.

We can now conclude the following lemma.
\begin{lemma}\label{lemma18}
Let $I$ be the context-free grammar given by~\eqref{Iw3}. For any $n\geqslant 1$, we have $$D_I^{n-1}(p_3)\mid_{p_1=x+y+z,p_2=1,p_3=xyz}=N_n(x,y,z).$$
In particular, $D_I(p_3)=p_1p_3$, $D_I^2(p_3)=p_1^2p_3+6p_2p_3^2$ and $D_I^3(p_3)=p_1^3p_3+24p_1p_2p_3^2+6p_3^3$.
\end{lemma}
\noindent{\bf A proof
Theorem~\ref{mainthm03}:}
\begin{proof}
Consider a change of the grammar $H$, which is given by~\eqref{G8def}.
Setting $w=p_3,v=p_1$ and $u=3p_2$, we get
$$D_H(p_3)=p_1p_3,~D_H(p_1)=6p_2p_3,D_H(p_2)=p_3,$$
which yield the grammar $I$.
It follows from~\eqref{Chen22} that
$$D_I^{n-1}(p_3)=D_H^{n-1}(w)\mid_{w=p_3,v=p_1,u=3p_2}=\sum_{i+2j+3k=2n+1}\gamma_{n,i,j,k}3^ip_2^ip_1^jp_3^k.$$
By~\eqref{G2} and Lemma~\ref{lemma18}, we obtain
$$N_n(x,y,z)=\sum_{i+2j+3k=2n+1}3^i\gamma_{n,i,j,k}(x+y+z)^j(xyz)^k.$$
This completes the proof of Theorem~\ref{mainthm03}.
\end{proof}
\section{The $e$-positivity of multivariate $k$-th order Eulerian polynomials}\label{section6}
A bivariate version of the Eulerian polynomial over the symmetric group is given as follows:
$$A_n(x,y)=\sum_{\pi\in\msn}x^{\asc(\pi)}y^{\des(\pi)}.$$
Clearly, $A_n(x,1)=A_n(1,x)=A_n(x)$.
Carlitz and Scoville~\cite{Carlitz74} showed that
\begin{equation*}\label{CarlitzSco74}
A_{n+1}(x,y)=xy\left(\frac{\partial}{\partial x}+\frac{\partial}{\partial y}\right)A_n(x,y),~A_1(x,y)=xy.
\end{equation*}
Foata and Sch\"utzenberger~\cite{Foata70} found that $A_n(x,y)$ has the gamma-expansion
\begin{equation*}\label{Anx-gamma}
A_n(x,y)=\sum_{k=1}^{\lrf{({n+1})/{2}}}\gamma(n,k)(xy)^k(x+y)^{n+1-2k},
\end{equation*}
where $\gamma(n,k)$ counts permutations in $\msn$ with $k$ descents, but with no double descents.

In this section, we always let $k$ be a given positive integer. A $k$-Stirling permutation of order $n$ is a multiset permutation of $\{1^k,2^k,\ldots,n^k\}$
with the property that all elements between two occurrences of $i$ are at least $i$, where $i\in [n]$, see~\cite{Lin21,Liu21,Yan2022} for the recent study on $k$-Stirling permutations and their variants.
Let $\mqn(k)$ be the set of $k$-Stirling permutations of order $n$. It is clear that $\mqn(1)=\msn,~\mqn(2)=\mqn$.

Let $\sigma\in\mqn(k)$.
The ascents, descents and plateaux of $\sigma$ of are defined as before, where we always set $\sigma_0=\sigma_{kn+1}=0$.
More precisely, an index $i$ is called an ascent (resp. descent, plateau) of $\sigma$ if $\sigma_i<\sigma_{i+1}$ (resp. $\sigma_i>\sigma_{i+1}$, $\sigma_i=\sigma_{i+1}$).
It is clear that $\asc(\sigma)+\des(\sigma)+\plat(\sigma)=kn+1$.
As a natural refinement of ascents, descents and plateaux, Janson-Kuba-Panholzer~{\cite{Janson11} introduced the following definition, and
related the distribution of $j$-ascents, $j$-descents and $j$-plateaux in $k$-Stirling permutations
with certain parameters in $(k+1)$-ary increasing trees.
\begin{definition}[{\cite{Janson11}}]\label{Lemma-Stirling}
An index $i$ is called a $j$-plateau (resp.~$j$-descent,~$j$-ascent) if $i$ is a plateau (resp.~descent,~ascent) and there are exactly $j-1$ indices $\ell<i$ such that
$a_{\ell}=a_i$.
\end{definition}

Let $\operatorname{plat}_j(\sigma)$ be the number of $j$-plateaux of $\sigma$. For $\sigma\in\mqn(k)$, it is clear that
$\operatorname{plat}_j(\sigma)\leqslant k-1$.
\begin{example}
Consider the $4$-Stirling permutation $\sigma=111223333221$.
The set of $1$-plateaux is given by $\{1,4,6\}$, the set of $2$-plateaux is given by $\{2,7\}$, and the set of $3$-plateaux is given by $\{8,10\}$.
Thus $\operatorname{plat}_1(\sigma)=3$ and $\operatorname{plat}_2(\sigma)=\operatorname{plat}_3(\sigma)=2$.
\end{example}

The {\it multivariate $k$-th order Eulerian polynomials} $C_n(x_1,\ldots,x_{k+1})$ are defined by
$$C_n(x_1,x_2,\ldots,x_{k+1})=\sum_{\sigma\in\mqn(k)}{x_1}^{\operatorname{plat}_1(\sigma)}{x_2}^{\operatorname{plat}_2(\sigma)}\cdots {x_{k-1}}^{\operatorname{plat}_{k-1}(\sigma)}{x_{k}}^{\operatorname{des}(\sigma)}{x_{k+1}}^{\operatorname{asc}(\sigma)}.$$
In particular, when $x_1=z,~x_2=\cdots=x_{k-1}=0$, $x_k=y$ and $x_{k+1}=x$, the polynomials $C_n(x_1,x_2,\ldots,x_{k+1})$ reduce to
$C_n(x,y,z)$; when $x_1=x_2=\cdots=x_{k-1}=0$, $x_k=1$ and $x_{k+1}=x$, the polynomials $C_n(x_1,x_2,\ldots,x_{k+1})$ reduce to the Eulerian polynomials
$A_n(x)$.

In the following, we always set $\operatorname{X}_{k+1}=\{x_1,x_2,\ldots,x_{k+1}\}$.
Let $e_i$ be the $i$-th
elementary symmetric function associated with $\operatorname{X}_{k+1}$. In particular, $$e_0=1,~e_1=x_1+x_2+\cdots+x_{k+1},~e_k=\sum_{i=1}^k\frac{e_{k+1}}{x_i},~e_{k+1}=x_1x_2\cdots x_{k+1}.$$
\begin{lemma}\label{Stirling}
Let
$G_1=\{x_1 \rightarrow e_{k+1},~ x_2\rightarrow e_{k+1},\ldots,~ x_{k+1}\rightarrow e_{k+1}\}$ be a grammar,
where $e_{k+1}=x_1x_2\cdots x_{k+1}$.
For $n\geqslant 1$, one has $D_{G_1}^n(x_1)=C_n(x_1,x_2,\ldots,x_{k+1})$.
\end{lemma}
\begin{proof}
We shall show
that the grammar $G_1$ can be used to generate $k$-Stirling permutations.
We first introduce a grammatical labeling of $\sigma\in \mqn(k)$ as follows:
\begin{itemize}
  \item [\rm ($L_1$)]If $i$ is an ascent, then put a superscript label $x_{k+1}$ right after $\sigma_i$;
 \item [\rm ($L_2$)]If $i$ is a descent, then put a superscript label $x_k$ right after $\sigma_i$;
\item [\rm ($L_3$)]If $i$ is a $j$-plateau, then put a superscript label $x_j$ right after $\sigma_i$.
\end{itemize}
The weight of $\sigma$ is
defined as the product of the labels, that is $$w(\sigma)={x_1}^{\operatorname{plat}_1(\sigma)}{x_2}^{\operatorname{plat}_2(\sigma)}\cdots {x_{k-1}}^{\operatorname{plat}_{k-1}(\sigma)}{x_{k}}^{\operatorname{des}(\sigma)}{x_{k+1}}^{\operatorname{asc}(\sigma)}.$$
Recall that we always set $\sigma_0=\sigma_{kn+1}=0$.
Thus the index $0$ is always an ascent and the index $kn$ is always a descent.
Thus $\mq_1(k)=\{^{x_{k+1}}1^{x_1}1^{x_2}1^{x_3}\cdots 1^{x_k}\}$. There are $k+1$ elements in $\mq_2(k)$ and they can be labeled as follows, respectively:
$$^{x_{k+1}}1^{x_1}1^{x_2}\cdots 1^{x_{k-1}}1^{x_{k+1}}2^{x_1}2^{x_2}\cdots 2^{x_{k-1}}2^{x_k},$$
$$^{x_{k+1}}1^{x_1}1^{x_2}\cdots 1^{x_{k-2}}1^{x_{k+1}}2^{x_1}2^{x_2}\cdots 2^{x_{k-1}}2^{x_k}1^{x_k},~~\cdots$$
$$^{x_{k+1}}2^{x_1}2^{x_2}\cdots2^{x_{k-1}}2^{x_k}1^{x_1}1^{x_2}\cdots 1^{x_{k-1}}1^{x_k}.$$
Note that $D_{G_1}(x_1)=e_{k+1}$ and $D_{G_1}^2(x_1)=e_ke_{k+1}$.
Then the weight of the element in $\mq_1(k)$ is given by $D_{G_1}(x_1)$, and the sum of weights of the elements in $\mq_2(k)$ is given by $D_{G_1}^2(x)$.
Hence the result holds for $n=1,2$.
We proceed by induction on $n$.
Suppose we get all labeled permutations in $\mq_{n-1}(k)$, where $n\geqslant 3$. Let
$\sigma'$ be obtained from $\sigma\in \mq_{n-1}(k)$ by inserting the string $nn\cdots n$ with length $k$.
Then the changes of labeling are illustrated as follows:
$$\cdots\sigma_i^{x_j}\sigma_{i+1}\cdots\mapsto \cdots\sigma_i^{x_{k+1}}n^{x_1}n^{x_2}\cdots n^{x_k}\sigma_{i+1}\cdots;$$
$$\sigma^{x_k}\mapsto \sigma^{x_{k+1}}n^{x_1}n^{x_2}\cdots n^{x_k};~~\quad ^{x_{k+1}}\sigma \mapsto ^{x_{k+1}}n^{x_1}n^{x_2}\cdots n^{x_k}\sigma.$$
In each case, the insertion of the string $nn\cdots n$ corresponds to one substitution rule in $G_1$.
Then the action of $D_{G_1}$ on the set of weights of all elements in $\mq_{n-1}(k)$ gives the set of weights of all elements in $\mq_n(k)$.
Therefore, we get the desired description of $C_n(x_1,x_2,\ldots,x_{k+1})$.
\end{proof}

It should be noted that in~\cite{Janson11}, there is no explicit connection to the $k$-th order
Eulerian polynomials is brought up.
By combining an urn model for the exterior leaves of $(k+1)$-ary increasing trees and a bijection between $(k+1)$-ary increasing trees and $k$-Stirling permutations,
Janson-Kuba-Panholzer~\cite[Theorem~2,~Theorem~8]{Janson11} found that the variables in $C_n(x_1,x_2,\ldots,x_{k+1})$ are exchangeable.
We can now present the main result of this section.
\begin{theorem}\label{mainthm02}
For $n\geqslant 2$ and $k\geqslant n-2$, we have
\begin{equation}\label{main01}
C_n(x_1,x_2,\ldots,x_{k+1})=\sum \gamma(n;i_1,i_2,\ldots,i_n)e_{k-n+2}^{i_n}e_{k-n+3}^{i_{n-1}}\cdots e_k^{i_2} e_{k+1}^{i_1},
\end{equation}
where the summation is over all sequences $(i_1,i_2,\ldots,i_{n})$ of nonnegative integers such that
$i_1+i_2+\cdots+i_{n}=n$, $1\leqslant i_1\leqslant n-1$, $i_n=0$ or $i_n=1$. When $i_n=1$, one has $i_1=n-1$.
The coefficients $\gamma(n;i_1,i_2,\ldots,i_n)$ equal the numbers of 0-1-2-$\cdots$-k-(k+1) increasing plane
trees on $[n]$ with $i_j$ degree $j-1$ vertices for all $1\leqslant j\leqslant n$.
\end{theorem}
\begin{proof}
Let $G_1$ be the grammar given in Lemma~\ref{Stirling}. Consider a change of $G_1$.
Note that $D_{G_1}(x_1)=e_{k+1},~D_{G_1}(e_{i})=(k-i+2)e_{i-1}e_{k+1}$ for $1\leqslant i\leqslant k+1$. Thus we get a new grammar
\begin{equation}\label{grammar002}
G_{2}=\{x_1\rightarrow e_{k+1},~e_{i}\rightarrow (k-i+2)e_{i-1}e_{k+1}~\text{for $1\leqslant i\leqslant k+1$}\},
\end{equation}
Note that $G_{2}(x_1)=e_{k+1},~G_{2}^2(x_1)=e_ke_{k+1},~G_{2}^3(x_1)=e_k^2e_{k+1}+2e_{k-1}e_{k+1}^2$,
$$D_{G_2}^4(x_1)=e_k^3e_{k+1}+8e_{k-1}e_ke_{k+1}^2+6e_{k-2}e_{k+1}^3,$$
$$D_{G_2}^5(x_1)=e_k^4e_{k+1}+22e_k^2e_{k-1}e_{k+1}^2+16e_{k-1}^2e_{k+1}^3+42e_{k-2}e_ke_{k+1}^3+24e_{k-3}e_{k+1}^4.$$
By induction, we assume that
\begin{equation}\label{DG2x1}
G_{2}^n(x_1)=\sum \gamma(n;i_1,i_2,\ldots,i_n)e_{k-n+2}^{i_n}e_{k-n+3}^{i_{n-1}}\cdots e_k^{i_2} e_{k+1}^{i_1}.
\end{equation}
Note that
\begin{align*}
G_{2}^{n+1}(x_1)&=G_{2}\left(\sum \gamma(n;i_1,i_2,\ldots,i_n)e_{k-n+2}^{i_n}e_{k-n+3}^{i_{n-1}}\cdots e_k^{i_2} e_{k+1}^{i_1}\right)\\
&=\sum ni_n\gamma(n;i_1,i_2,\ldots,i_n)e_{k-n+1}e_{k-n+2}^{i_n-1}e_{k-n+3}^{i_{n-1}}\cdots e_k^{i_2} e_{k+1}^{i_1+1}+\\
&\sum (n-1)i_{n-1}\gamma(n;i_1,i_2,\ldots,i_n)e_{k-n+2}^{i_n+1}e_{k-n+3}^{i_{n-1}-1}\cdots e_k^{i_2} e_{k+1}^{i_1+1}+\cdots+\\
&\sum 2i_2\gamma(n;i_1,i_2,\ldots,i_n)e_{k-n+2}^{i_n}e_{k-n+3}^{i_{n-1}}\cdots e_{k-1}^{i_3+1}e_k^{i_2-1} e_{k+1}^{i_1+1}+\\
&\sum i_1\gamma(n;i_1,i_2,\ldots,i_n)e_{k-n+2}^{i_n}e_{k-n+3}^{i_{n-1}}\cdots e_k^{i_2+1} e_{k+1}^{i_1},
\end{align*}
which yields that the expansion~\eqref{DG2x1} holds for $n+1$. Combining Lemma~\ref{Stirling} and~\eqref{DG2x1}, we get~\eqref{main01}.
By induction, one can verify that $i_1+i_2+\cdots+i_n=n$, $1\leqslant i_1\leqslant n-1$, $i_n=1$ or $i_n=0$.

Using~\eqref{grammar002}, the combinatorial interpretation of
$\gamma(n;i_1,i_2,\ldots,i_n)$ can be proved along the same lines as the proof of~\cite[Theorem~4.1]{Chen21}.
However, we give a direct proof of it for our purpose.
Let $T$ be a 0-1-2-$\cdots$-k-(k+1) increasing plane tree on $[n]$. The labeling of
$T$ is given by labeling a degree $i$ vertex by $e_{k-i+1}$ for all $0\leqslant i\leqslant k+1$.
In particular, label a leaf by $e_{k+1}$ and label a degree $k+1$ vertex by $1$.
Let $T'$ be a 0-1-2-$\cdots$-k-(k+1)
increasing plane tree on $[n+1]$ by adding $n+1$ to $T$ as a leaf.
We can add $n+1$ to $T$ only as a child of a vertex $v$ that is not of degree $k+1$.
For $1\leqslant i\leqslant k+1$, if the vertex $v$ is a degree $k-i+1$ vertex with label $e_i$,
there are $k-i+2$ cases to attach $n+1$ (from left to right, say). In either case, in $T'$, the vertex $v$ becomes a
degree $k-i+2$ with label $e_{i-1}$ and $n+1$ becomes a leaf with label $e_{k+1}$. Hence
the insertion of $n+1$ corresponds to the substitution rule $e_{i}\rightarrow (k-i+2)e_{i-1}e_{k+1}$.
Therefore, $G_{2}(x_1)$ equals the sum of the weights
of 0-1-2-$\cdots$-(k+1) increasing plane trees on $[n]$, and the combinatorial interpretation of $\gamma(n;i_1,i_2,\ldots,i_n)$ follows. This completes the proof.
\end{proof}

By using $G_{2}^{n+1}(x_1)=G_{2}\left(G_{2}^{n}(x_1)\right)$, it is routine to verify that
\begin{align*}
&\gamma(n+1;1,n,0\ldots,0)=\gamma(n;1,n-1,0,\ldots,0)=1,\\
&\gamma(n+1;n,0,\ldots,0,1)=n\gamma(n;n-1,0,\ldots,0,1)=n!,\\
&\gamma(n+1;i_1,i_2,\ldots,i_n,0)=i_1\gamma(n;i_1,i_2-1,i_3,\ldots,i_n)+\\
&\sum_{j=2}^{n-1}j(i_j+1)\gamma(n;i_1-1,i_2,\ldots,i_{j-1},i_j+1,i_{j+1}-1,i_{j+2}\ldots,i_n).
\end{align*}

Note that $\gamma(3;2,0,1,0,\ldots,0)=2,~\gamma(4;2,1,1,0,\ldots,0)=8$ and
$$\gamma(n+1;2,n-2,1,0,\ldots,0)=2\gamma(n;2,n-3,1,0,\ldots,0)+2(n-1)\gamma(n;1,n-1,0,\ldots,0).$$
By induction, it is easy to verify that
\begin{equation}\label{gamman3}
\gamma(n;2,n-3,1,0,\ldots,0)=2^n-2n~\text{for $n\geqslant 3$}.
\end{equation}
Recall that the second-order Eulerian numbers $C_{n,j}$ satisfy the recurrence relation
\begin{equation*}\label{secondEu-recu}
C_{n+1,j}=jC_{n,j}+(2n+2-j)C_{n,j-1},
\end{equation*}
with the initial conditions $C_{1,1}=1$ and $C_{1,j}=0$ if $j\neq 1$ (see~\cite{Bona08,Gessel78}). In particular,
\begin{equation*}
C_{n,2}=2^{n+1}-2(n+1).
\end{equation*}
Comparing this with~\eqref{gamman3}, we see that
$\gamma(n;2,n-3,1,0,\ldots,0)=C_{n-1,2}$ for $n\geqslant 3$.
Following Janson~\cite{Janson08}, the number $C_{n,j}$ equals the number of increasing plane trees on $[n+1]$ with $j$ leaves.
So we immediately get the following result.
\begin{proposition}
For $n\geqslant 2$ and $1\leqslant j\leqslant n-1$, we have
$$C_{n-1,j}=\sum_{i_2+i_3+\cdots+i_{n}=n-j} \gamma(n;j,i_2,\ldots,i_{n-1},i_n).$$
\end{proposition}
\bibliographystyle{amsplain}

\begin{thebibliography}{10}






%
%

\bibitem{Bona08}
M. B\'ona, \textit{Real zeros and normal distribution for statistics on Stirling permutations defined by Gessel and Stanley}, SIAM J. Discrete Math., \textbf{23} (2008/09), 401--406.





\bibitem{Bre94}
F. Brenti, \textit{$q$-Eulerian polynomials arising from Coxeter groups}, European J. Combin., \textbf{15} (1994), 417--441.


\bibitem{Buckholtz}
J.D. Buckholtz, \textit{Concerning an approximation of Copson}, Proc. Amer. Math. Soc., 14 (1963), 564--568.


\bibitem{Carlitz65}
L. Carlitz, \textit{The coefficients in an asymptotic expansion}, Proc. Amer. Math. Soc., \textbf{16} (1965) 248--252.
%
%
%
%
%
%
\bibitem{Carlitz74}
L. Carlitz, R. Scoville, \textit{Generalized Eulerian Numbers: Combinatorial applications}, J. Reine Angew. Math., \textbf{265} (1974), 110--137.
%
%
%
%
\bibitem{Chen93}
W.Y.C. Chen, \textit{Context-free grammars, differential operators and formal
power series}, Theoret. Comput. Sci., \textbf{117} (1993), 113--129.

\bibitem{Chen17}
W.Y.C. Chen, A.M. Fu, \textit{Context-free grammars for permutations and increasing trees}, Adv. in Appl. Math., \textbf{82} (2017), 58--82.

\bibitem{Chen21}
W.Y.C. Chen, A.M. Fu, \textit{A context-free grammar for the $e$-positivity of the trivariate second-order Eulerian polynomials},
Discrete Math., \textbf{345}(1) (2022), 112661.
%
\bibitem{Chen2102}
W.Y.C. Chen, R.X.J. Hao and H.R.L.Yang, \textit{Context-free grammars and multivariate stable polynomials over Stirling permutations}, In: V. Pillwein and
C. Schneider (eds.), Algorithmic Combinatorics: Enumerative Combinatorics, Special Functions and Computer Algebra, pp. 109--135, Springer, 2021.


\bibitem{Chow08}
C.-O. Chow, \textit{On certain combinatorial expansions of the Eulerian polynomials}, Adv. in Appl. Math., \textbf{41} (2008), 133--157.

%
%
\bibitem{Dumont80}
D. Dumont, \textit{Une g\'en\'eralisation trivari\'ee sym\'etrique des nombres eul\'eriens}, J. Combin. Theory Ser. A, \textbf{28} (1980), 307--320.
%
\bibitem{Dumont96}
D. Dumont, \textit{Grammaires de William Chen et d\'erivations dans les arbres et
arborescences}, S\'em. Lothar. Combin., \textbf{37}, Art. B37a (1996), 1--21.
%
\bibitem{Dzhumadil14}
A. Dzhumadil'daev, D. Yeliussizov, \textit{Stirling permutations on multisets}, Europ. J. Combin., \textbf{36} (2014), 377--392.
%
\bibitem{Foata70}
D. Foata, M. P. Sch\"utzenberger, \textit{Th\'eorie g\'eometrique des polyn\^omes eul\'eriens}, Lecture Notes in Math., vol. 138, Springer, Berlin, 1970.

%
%
\bibitem{Gessel78}
I. Gessel, R.P. Stanley, \textit{Stirling polynomials}, J. Combin. Theory Ser. A, \textbf{24} (1978), 25--33.
%
\bibitem{Haglund12}
J. Haglund and M. Visontai, \textit{Stable multivariate Eulerian polynomials and
generalized Stirling permutations}, European J. Combin., \textbf{33} (2012), 477--487.
%
\bibitem{Hwang20}
H.-K. Hwang, H.-H. Chern, G.-H. Duh, \textit{An asymptotic distribution theory for Eulerian
recurrences with applications}, Adv. in Appl. Math., \textbf{112} (2020), 101960.

\bibitem{Janson08}
S. Janson, \textit{Plane recursive trees, Stirling permutations and an urn model},
Proceedings of Fifth Colloquium on Mathematics and Computer Science,
Discrete Math. Theor. Comput. Sci. Proc., vol. AI, 2008, pp. 541--547.
%
\bibitem{Janson11}
S. Janson, M. Kuba, A. Panholzer, \textit{Generalized Stirling permutations, families of increasing trees and urn models}, J. Combin.
Theory Ser. A, \textbf{118}(1) (2011), 94--114.

\bibitem{Koganov96}
L.M. Koganov, \textit{Universal bijection between Gessel-Stanley permutations and diagrams of connections of corresponding ranks}, Russ
Math. Surv., \textbf{51}(2) (1996), 333--335.



\bibitem{Lin2101}
Z. Lin, J. Ma, S.-M. Ma, Y. Zhou, \textit{Weakly increasing trees on a multiset}, Adv. Appl. Math., \textbf{129} (2021) 102206.


\bibitem{Lin21}
Z. Lin, J. Ma, P.B. Zhang, \textit{Statistics on multipermutations and partial $\gamma$-positivity}, J. Comb. Theory, Ser. A,
\textbf{183} (2021), 105488.


\bibitem{Liu21}
S.H. Liu, \textit{MacMahon's equidistribution theorem for $k$-Stirling permutations}, Adv. in Appl. Math., \textbf{128} (2021), 102193.
%
\bibitem{Ma15}
S.-M. Ma, T. Mansour, \textit{The $1/k$-Eulerian polynomials and $k$-Stirling permutations}, Discrete Math., \textbf{338} (2015), 1468--1472.

\bibitem{MaYeh17}
S.-M. Ma, Y.-N. Yeh, \textit{Eulerian polynomials, Stirling permutations
of the second kind and perfect matchings}, Electron. J. Combin., \textbf{24}(4) (2017), \#P4.27.

\bibitem{Ma19}
S.-M. Ma, J. Ma, Y.-N. Yeh, \textit{The ascent-plateau statistics on Stirling permutations}, Electron. J. Combin., \textbf{26}(2) (2019), \#P2.5.


\bibitem{Ma1902}
S.-M. Ma, J. Ma, Y.-N. Yeh, \textit{$\gamma$-positivity and partial $\gamma$-positivity of descent-type polynomials}, J. Combin. Theory Ser. A, \textbf{167} (2019), 257--293.
%
%
\bibitem{Ma2021}
S.-M. Ma, J. Ma, Y.-N. Yeh, R.R Zhou, \textit{Jacobian elliptic functions and a family of bivariate peak polynomials}, European J. Combin., \textbf{97} (2021), 103371.
%

\bibitem{Ramanujan27}
S. Ramanujan, \textit{Collected papers}, p.~26, Cambridge, 1927.

\bibitem{Riordan76}
J. Riordan, \textit{The blossoming of Schr\"{o}der's fourth problem}, Acta Math., \textbf{137} (1976), 1--16.

\bibitem{Savage1202}
C.D. Savage and G. Viswanathan, \textit{The $1/k$-Eulerian polynomials}, \newblock {\em Electron. J. Combin.}
\textbf{19} (2012), \#P9.

\bibitem{Park1994}
S.K. Park, \textit{The $r$-multipermutations}, J. Combin. Theory Ser. A, \textbf{67}(1)(1994), 44--71.

\bibitem{Yan2022}
S.H.F. Yan, Y. Huang, L. Yang, \textit{Partial $\gamma$-positivity for quasi-Stirling permutations of multisets}, Discrete Math., \textbf{345} (2022), 112742.

\bibitem{Zhuang17}
Y. Zhuang, \textit{Counting permutations by runs}, J. Combin. Theory Ser. A, \textbf{142} (2016), 147--176.
\end{thebibliography}

\end{document}